\documentclass{ws-m3as}
\usepackage{hyperref}
\hypersetup{
    colorlinks=true,
    linkcolor=blue,
    filecolor=magenta,      
    urlcolor=cyan,
}
\def\R{\mathbb{R}}

\def\P{{\mathbb P}}     
\def\E{{\mathbb E}}     
\def\<{{\langle}} 
\def\>{{\rangle}} 
\DeclareMathOperator{\erf}{erf}

\begin{document}

\markboth{Fan et al.}{Hitting Time of Rapid Intensification Onset}

%
\catchline{}{}{}{}{}
%

\title{Hitting Time of Rapid Intensification Onset \\in Hurricane-like Vortices}


\author{Wai-Tong (Louis) Fan
}
\address{Department of Mathematics, Indiana University\\
Bloomington, IN 47405, USA\\
waifan@iu.edu}

\author{Chanh Kieu}
\address{Department of Earth and Atmospheric science, Indiana University\\
Bloomington, IN 47405, USA\\
ckieu@indiana.edu}

\author{Dimitrios Sakellariou}
\address{Department of Mathematics, Indiana University\\
Bloomington, IN 47405, USA\\
disakell@iu.edu}

\author{Mahashweta Patra}
\address{Department of Earth and Atmospheric science, Indiana University\\
Bloomington, IN 47405, USA\\
mpatra@iu.edu}

\maketitle

\begin{history}
\received{(Day Month Year)}
\revised{(Day Month Year)}
\comby{(xxxxxxxxxx)}
\end{history}

\begin{abstract}
Predicting tropical cyclone (TC) rapid intensification (RI) is an important yet challenging task in current operational forecast due to our incomplete understanding of TC nonlinear processes. This study examines the variability of RI onset, including  the probability of RI occurrence and the timing of RI onset, using a low-order stochastic model for TC development. Defining RI onset time as the first hitting time in the model for a given subset in the TC-scale state space, we quantify the probability of the occurrence of RI onset and the  distribution of the timing of RI onset for a range of initial conditions and model parameters. Based on  asymptotic analysis for stochastic differential equations, our results show that RI onset occurs later, along with a larger variance of RI onset timing, for weaker vortex initial condition and stronger noise amplitude. In the small noise limit, RI onset probability approaches one and the RI onset timing has less uncertainty (i.e., a smaller variance), consistent with observation of TC development under idealized environment. Our theoretical results are verified against Monte-Carlo simulations and compared with explicit results for a general 1-dimensional system, thus providing new insights into the variability of RI onset and helping better quantify the uncertainties of RI variability for practical applications.
\end{abstract}

%
%
\section {Introduction}
Rapid intensification (RI) is an inherent feature of hurricane (also known as tropical cyclone or TC) dynamics by which a TC intensifies quickly in a very short period of time \footnote{Practically, RI is defined as a change of the maximum 10-m wind of 30 kt (~15 ms$^{-1}$ in 24 hours).} Predicting RI onset is therefore of great importance in operational TC forecast such that proper and timely risk management and preparation can be initiated \cite{KaplanDemaria2003,Sampson_etal2011,Rappaport_etal2012,Fisher_etal2019}.

While RI is an inherent property of TC development that is guaranteed to occur under idealized environment, the probability or the exact moment that RI onset takes place in real-time forecast highly fluctuates as a result of varying environmental conditions \cite{KowchEmanuel2015,NguyenChanhFan}. Despite some progress in improving TC intensity forecast skill, RI prediction has been challenging. As shown in, e.g., \refcite{Kaplan_etal2015,Tallapragada_etal2015,Tallapragada_etal2014a,Rozoff_etal2015,Yang2016}, current operational models still have a high false alarm rate and a moderate  probability of detection for RI prediction, even at a short 24-36 hour lead time. The RI forecast skill is significantly deteriorated as the forecast lead time is extended longer, making it hard to reliably predict RI in real-time applications. With various uncertainties in TC intensity fluctuation related to vortex initial conditions, model errors, boundary conditions as well as potential existence of TC intensity chaotic dynamics and random variability \cite{majda1999models,KowchEmanuel2015,KieuMoon2016,NguyenChanhFan}, it is necessary to examine to what extent RI onset can be best predicted for future operational applications and model improvement.     

From the practical perspective, TC development is an intrinsically random process that can never be fully controlled due to the stochastic nature of the atmosphere. Naturally, one then expects RI onset to be impacted by such random variability from the atmosphere, especially during the early stage of TC development that possesses high uncertainties in both the structure and strength. Figure \ref{fig1} shows an example of TC intensity evolution obtained from the Coupled Ocean Atmospheric Prediction System (COAMPS-TC) model, using an ensemble of simulations with small random noises centered on a given initial condition \cite{Kieu_etal2021}. One notices  that the RI onset timing in this ensemble, defined to be the first moment in the model simulation that the maximum surface wind (Vmax) increases by 14.5 $m\;s^{-1}$ (30 kt) per 24 hr, is not a deterministic variable but varies significantly, regardless of how perfect environmental conditions are. 
%
%
\begin{figure}[tbh]
\centering
\includegraphics[scale=0.5]{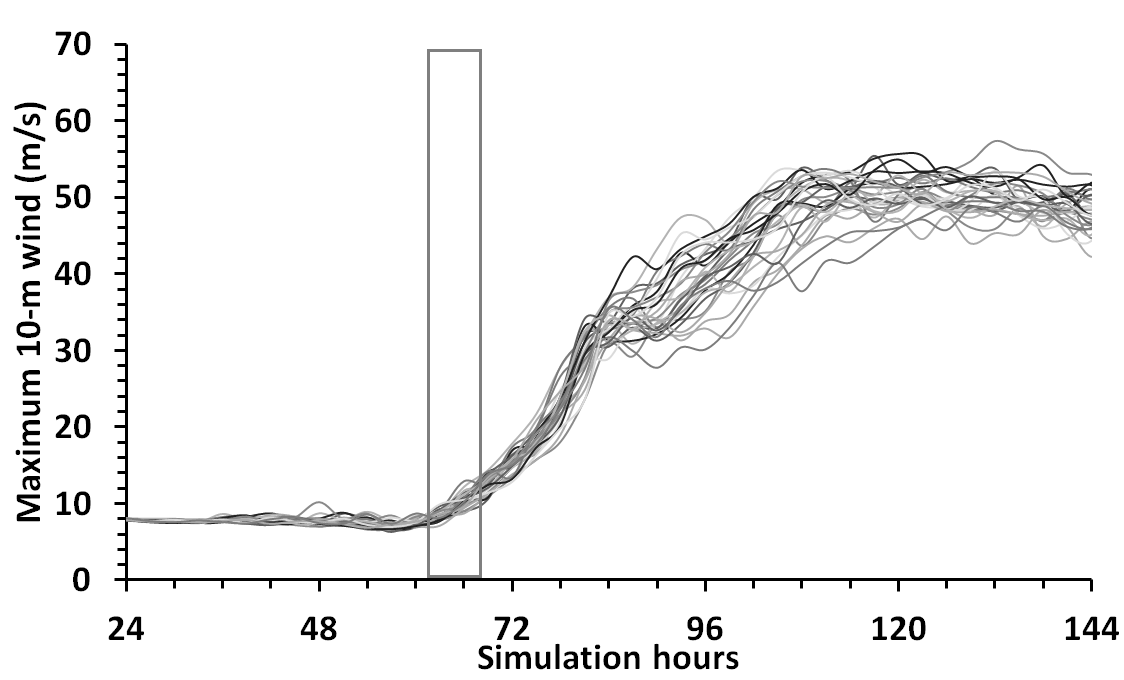}
\caption{Time series of the maximum 10-m wind (VMAX) ensemble during the course of idealized simulations using the COAMPS-TC model, under a perfect model scenario. The ensemble is perturbed by small random perturbations at the initial condition as presented in \protect\refcite{Kieu_etal2021}. The gray box denotes the interval at which RI onset time varies among different ensemble members. Here, the RI onset moment is defined as the time into simulation that the VMAX change in the next 24 hours is $\ge 14.5 \,ms^{-1}$.}

\label{fig1}
\end{figure}

The above variation of RI onset timing as illustrated by the COAMPS-TC model due to random noise is in fact just one among many other possible sources of uncertainties related to, for example, boundary and surface layer paramterization, model physics, vortex initial conditions and locations, or potential existence of TC chaotic dynamics \cite{NavarroHakim2016,Rasp_etal2018,Bryan_etal2017,KieuMoon2016,Kieu_etal2018}. The combination of all these uncertainties apparently indicates that RI onset should not be treated as a deterministic but more as a stochastic process. How to  quantify the probability of RI onset as well as its timing are, nevertheless, open questions in the current TC research. 

In this paper, based on asymptotic hitting analysis for stochastic differential equations (SDE), we study the probability of RI onset as well as the variability of RI onset timing. Our objective is to examine RI onset under 
idealized conditions such that intrinsic random characteristics of RI onset in the absence of all environmental asymmetries can be investigated. For this purpose, the \textit{first hitting time} (also known as the \textit{first passage time}) technique for stochastic processes appears to be appropriate and beneficial due to its connection with stochastic analysis.
Defined as the moment when a stochastic process first visits a given subset in the state space, the first-hitting time can be directly linked to RI onset time from which the first-hitting time techniques can be applied to study the variability of RI onset as expected. 


To the best of our knowledge, this approach and its applications to TC development have not been previously explored. As such, we wish to present in this study a theoretical framework that could allow one to rigorously quantify the probability of RI onset as well as the variability of RI onset timing as a function of the ambient environment and TC initial conditions.


The rest of this study is organized as follows. In the next section, a stochastic model for TC development is presented, followed by a formal definition of RI onset within the first hitting-time framework. Section 3 presents theoretical results for the probability of RI occurrence and the distribution of RI onset time. Monte-Carlo simulations to verify our theoretical results will be provided in Section 4, along with additional insights on the dependence of RI onset on model parameters. Concluding remarks are given in the final section.
%
%
\section{Formulation}
\subsection {Stochastic model for TC intensification}
Under the axisymmetric assumption for TC development, Kieu and Wang (2017, \cite{KieuWang2017a}) presented a simple low-order model that is based on a few fundamental scales of TCs. Unlike common TC balance models, this TC-scale dynamics, which is a modified version of a TC model originally proposed by \refcite{Kieu2015} and is hereinafter referred to as the MSD model, is time-dependent and explicitly contains the maximum potential intensity limit as one of its critical points. In the nondimensional form, the MSD system in \refcite{KieuWang2017a} can be summarized as follows
\begin{equation}\label{MSD}
\begin{cases}
\frac{du}{dt}=pv^2-(p+1)b - u|v|\\
\frac{dv}{dt}=-uv - v |v|\\
\frac{db}{dt}=bu+su+|v|-rb.
\end{cases},
\end{equation}
where $(u,v,b)$ denote non-dimensional variables that represent the maximum radial wind, the maximum tangential wind, and the warm core anomaly in the TC inner-core region. The parameter $p$ is proportional to the squared ratio of the depth of the troposphere over the depth of the boundary layer, $s$ is an effective tropospheric static stability parameter, and $r$ represents the Newtonian cooling. Detailed derivation of this TC-scale system under the assumption of wind induced surface heat exchange (WISHE) feedback can be found in \cite{KieuWang2017a}.

Because of the dependence of frictional forcing and the WISHE feedback on the wind amplitude, it should be noted that the absolute sign in Eq. \eqref{MSD} results in two possibilities for TC development corresponding to cyclonic and anticyclonic systems. To ease our subsequent analyses, we will focus only on the regime in the state space where $v>0$, which corresponds to cyclonic TCs in the Northern Hemisphere. This cyclonic system will be hereinafter explicitly referred to as an MSD$_+$ system (see Eqs. (69)-(71) in \refcite{KieuWang2017a}), which is described by the following equations,
\begin{equation}\label{Eg1}
\begin{cases}
\frac{du}{dt}=pv^2-(p+1)b-uv\\
\frac{dv}{dt}=-uv-v^2\\
\frac{db}{dt}=bu+su+v-rb
\end{cases}.
\end{equation}
To simplify our notation, we write this MSD$_+$ system  in the form
\begin{equation}\label{ODE_MSD}
\frac{dx(t)}{dt} = \mu\big(x(t)\big),\quad t \geq 0, 
\end{equation}
where $x(t) \equiv (u(t), v(t), b(t))$, and the vector field $\mu=(\mu_1,\mu_2,\mu_3):\,\R^3 \to \R^3$ is the forcing function of \eqref{Eg1}, i.e., 

\begin{equation}\label{SDE_UVB3abc}
\begin{cases}
    \mu_1(u,v,b)=&\,pv^2-(p+1)b - u v \\
    \mu_2(u,v,b)=&\,-uv-v^2\\
    \mu_3(u,v,b)=&\,bu+su+v-rb
\end{cases}.
\end{equation}

While the low-order MSD system is admittedly simple as compared to real TCs, the fact that the dynamics of TC development can be formulated in such a mathematically closed form is critical here. This is because this MSD system  allows one to obtain different insights  into the underlying mechanisms of TC development beyond numerical simulations by full-physics models that one cannot fully control. 

Given the above deterministic model \eqref{Eg1} for TC development, we next extend it to a stochastic system. Following \refcite{NguyenChanhFan}, stochastic forcing is introduced to the MSD system as an additive Wiener process. Specifically, we consider the stochastic process $X_t:=(U_t,V_t,B_t)$ solving the time-homogeneous It\^o stochastic differential equation as follows
\begin{equation}\label{SDE_UVB}
dX_t=\mu(X_t)\,dt +\epsilon\,dW_t,  \quad t \geq 0, 
\end{equation}
where $\epsilon>0$ is a constant (the diffusion coefficient) that parametrizes the magnitude of the fluctuation of the random forcing, and $W$ is a standard $3$-dimensional Wiener processes. Explicitly,
\begin{equation}\label{SDE_UVB2}
\begin{cases}
dU_t= (p V_t^2-(p+1)B_t-U_t V_t)\,dt \,+ \epsilon\, dW^{(u)}_t\\
dV_t=(-U_tV_t-V^2_t)\,dt \,+ \epsilon\, dW^{(v)}_t\\ 
dB_t= (B_t U_t+ s U_t +V_t-rB_t)\,dt \,+ \epsilon\, dW^{(b)}_t
\end{cases},
\end{equation}
where $\{W^{(u)},\,W^{(v)},W^{(b)}\}$ are independent Wiener processes. The use of these independent Wiener processes to represent the random forcing for the MSD system significantly simplifies the problem both theoretically and numerically. For example, a numerical solution to equation \eqref{SDE_UVB2} with a sufficiently small discretization time step $\Delta t$ can be obtained by using the simple Euler-Maruyama scheme in which a Gaussian random variable with variance $(\Delta t) \epsilon^2$ is added to each state variable in every iteration \cite{budhiraja2017uniform,NguyenChanhFan}. Figure \ref{fig2} shows an illustration of numerical simulations of the MSD system \eqref{SDE_UVB2} for 50 different realizations, using the same method and parameters as in \refcite{NguyenChanhFan}. One notices apparently from this result that the MSD system displays RI for many realizations, while a few realizations quickly decay. For those that display RI, notice also that the RI onset timing varies as well (see the crosses in Figure \ref{fig2}). How the probability of RI occurrence and its related variability depend on model parameters or vortex initial conditions is the main question we wish to tackle herein. The closed form of the MSD system as given by \eqref{SDE_UVB2} is in this regard very noteworthy, as one can employ powerful mathematical tools such as stochastic calculus and asymptotic analysis to study the It\^o SDE \eqref{SDE_UVB2}. The rigorous meaning and the well-poseness of these equations can be found in standard textbooks such as \refcite{MR1121940} and will not be discussed further here.  
%
%
\begin{figure}[tbh]
\centering
\includegraphics[scale=0.4]{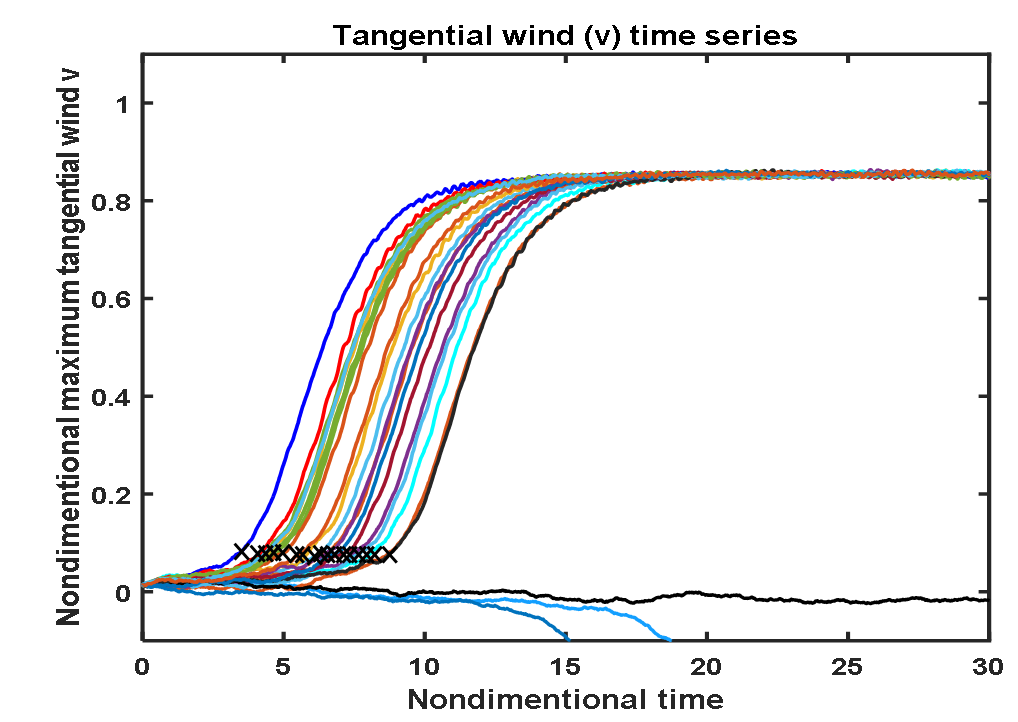}
\caption{
Time series of 30 realizations of the $v$ component as obtained from numerical integration of the MSD system, using the same set of parameters as in \protect\refcite{NguyenChanhFan}. The crosses denote the RI onset moment, which is defined as the first time into simulation that $v$ component rapidly changes similar to that in Figure 1.}
\label{fig2}
\end{figure}

\subsection{RI onset definition}
To assess the occurrence (and failure) of RI onset and to quantify the distribution of RI onset timing, it is necessary to introduce a formal definition of RI onset such that theoretical analyses can be obtained. Given the previous studies on the first-hitting time for stochastic systems, we herein define RI onset time as the first moment that the $V$ component reaches a given level $\ell\in (0,\infty)$ afterward RI will certainly occur. Previous analyses of the MSD system showed that such an RI onset level always exist \cite{Kieu2015}, because the MSD system contains a single stable point at the maximum potential intensity limit (cf. Figures \ref{fig1} or \ref{fig2} and \cite{NguyenChanhFan}). As such, when $V$ reaches the level $\ell$, TC  intensification is ensured to rapidly approach the potential intensity state, thus justifying our definition of RI onset here. 

\noindent Given the above definition of RI onset time, we now investigate the following two specific questions: 
\begin{description}
\item[1. Probability of RI onset occurrence:] whether or not the maximum tangential wind (i.e., $v$) will reach the level $\ell$ such that RI onset can occur, starting from a given initial condition; and
\item[2. Variability of RI onset timing:] if RI occurs, what is the statistical distribution of RI onset time?
\end{description}
To be more specific for our subsequent analyses, we introduce the following hitting times for the stochastic MSD system \eqref{SDE_UVB2}:
\begin{align}
    \tau_+&:=\inf\{t\geq 0:\; V_t =\ell\} \quad\text{the first time when $V$ reaches }\ell. \label{tau+}\\
    \tau_0&:=\inf\{t\geq 0:\; V_t=0\} \quad\text{the first time when $V$ crosses zero level} \notag\\
    & \qquad\qquad \qquad \qquad \qquad \quad\text{(i.e., an initial vortex dies out).} \label{tau0}
\end{align}
From the above definitions, the \textit{RI onset time} for the SDE \eqref{SDE_UVB2} is $\tau_+$. Furthermore, we say that \textit{RI onset occurred }if $\tau_+<\tau_0$. That is, when the trajectory of $v$ hits $\ell$ without dying out before then.  The condition $\tau_+ <\tau_0$ is needed here, any tropical disturbance hitting the level $v=0$ will be considered as being dissipated and so there is no RI for this vortex development in reality.

Due to the stochastic nature of the TC stochastic dynamics, it is apparent that $\tau_+$ is a random variable. As such, our aim is to obtain the probability  $\mathbb{P}_{x_0}(\tau_+<\tau_0)$ as a function of the initial condition $x_0=(u_0,v_0,b_0)$ and the parameters of the MSD system \eqref{SDE_UVB2}. Since the initial point can be any point in the state space, we define the function
\begin{equation}\label{Def:px}
    p(x):=\mathbb{P}_{x}(\tau_+<\tau_0)
\end{equation}
and we view $x=(u,v,b)$ as a generic point in the state space interchangeably. 

For comparison between the stochastic and deterministic MSD system, one needs also a deterministic RI onset time $T$ for the ODE \eqref{Eg1}, which is defined as follows
\begin{align}\label{Def:T_v}
T:=\inf\{t\geq0 \,; v(t) = \ell\}
\end{align}
which is the first time the trajectory $(v(t))_{t\geq 0}$ of \eqref{Eg1} hits level $\ell$ in the absence of all stochastic forcings. It should be noted that one cannot choose a too large value for $\ell$, as the MSD system possesses a unique stable point $v_*$ \cite{Kieu2015,KieuWang2017a}. Thus, $v$ will be always attracted to its equilibrium $v_*$ and never reach $\ell$ if $\ell$ is set too large. Hence, it is natural to make the following mild assumption throughout the paper.
\begin{assumption}\label{A:ell}
$0< \ell < v_*$, where $v_*$ is the $v$-component of the stable critical point in the phase space $(u,v,b)$ of the ODE \eqref{Eg1}.
\end{assumption}
Under this assumption, $T$ is finite (i.e. the $v$-component of the ODE must hit level $\ell$) if the initial condition $v_0 >0$ in contrast to the hitting time in a SDE system. That is, the $v$-trajectory for the ODE never hits zero so long as the initial value $v_0$ is positive as proven in \refcite{KieuWang2017a}.
%
%
\section{Theoretical results}\label{S:Theoretical}
In this section, we present rigorous analyses of the RI onset probability $p(x)=\mathbb{P}_x(\tau_+<\tau_0)$ for the stochastic model \eqref{SDE_UVB2}, along with the conditional probability distribution of $\tau_+$.
To obtain further analytic insight, the $v$-component of \eqref{SDE_UVB2} will be compared with the behavior of a general 1-dimensional system in Section \ref{S:1-dim}.

Recall from our Definition \eqref{tau+} of RI onset time that $\tau_+$ is the first time for $V$ to reach level $\ell$. Thus, we will apply asymptotic techniques for SDE to estimate the hitting time in the MSD model \eqref{SDE_UVB2}.
Note first that Eq. \eqref{SDE_UVB2} has a unique strong solution $X=(X_t)_{t\geq 0}$, where $X_t=(U_t,V_t,B_t)$ is a three-dimensional vector for each time $t\geq 0$. Furthermore, $X$ is a continuous-time strong Markov process with  infinitesimal generator $\mathcal{L}$ defined by the following differential operator
\begin{equation}\label{L1;Generator}
    \mathcal{L}f(x)=
  \mu_1\frac{\partial f}{\partial u}+
    \mu_2\frac{\partial f}{\partial v}+
    \mu_3\frac{\partial f}{\partial b}+ \frac{\epsilon^2}{2}  \Bigg(\frac{\partial^2 f}{\partial u^2}+\frac{\partial^2 f}{\partial v^2}+\frac{\partial^2 f}{\partial b^2}\Bigg),
\end{equation}
where $x=(u,v,b)$ is a generic point in the state space of the SDE. Let $p(x,t)$ be the probability density for $X_t$; that is, $\P(X_t\in dx)=p(x,t)dx$ where $dx$ is the Lebesque measure in $\R^d$. Then $p(x,t)$ satisfies the Fokker–Planck equation
$\frac{\partial p(x,t)}{\partial t} = \mathcal{L}^*p(x,t)$,
where $\mathcal{L}^*$ is the adjoint of $\mathcal{L}$ in the Hilbert space $L^2(dx)$. 
These facts follow from standard techniques in stochastic calculus as can be seen, for instance, in \cite[Chapter 5]{MR1121940}. 

In principle, \eqref{L1;Generator} enables one to obtain all desired statistics of RI onset time. However, due to nonlinearity of the SDE there is no explicit formula for the density $p(x,t)$. In the next two subsections, we shall therefore derive formal results for the probability of RI onset and the distribution of $\tau_+$ in the asymptotic limit of small stochastic forcing. These formal connections between SDE and PDE are the starting point of more in-depth analysis of the statistics of RI onset time that we can later verify by Monte-Carlo simulations. 
\begin{remark}[Implicit dependence parameters ]\rm Note that the process $X=(U,V,B)$, the generator $\mathcal{L}$, the onset time $\tau_+$, and the extinction time $\tau_0$ all depend on the noise parameter $\epsilon$ and the MSD model parameters $(p,r,s)$. 
These dependence are made implicit to simplify notation.
\end{remark}

\subsection{Probability of RI onset}\label{S:probOnset}
In the case when the initial value of $V$ is positive (i.e. $V_0>0$), it is possible to obtain a simplification for the MSD$_{+}$ system based on the fact that RI onset would not occur if $V$ hits zero level or becomes negative (i.e., an anticyclonic vortex). We therefore begin with the following lemma that expresses the probability of RI onset occurrence. That is, the probability that $V$ reaches a prescribed level $\ell>0$ without dying out. Practically, this probability indicates the development of a cyclonic vortex ($v>0$) instead of anticyclonic vortex ($v<0$), given that the initial state of the vortex is cyclonic in the Northern hemisphere (or an anticyclonic vortex from an initial anticyclonic state in the Southern hemisphere).

\begin{lemma}[Probability of RI onset]\label{Lm:HittingProb}
Let $p(x)= \P_x(\tau_+<\tau_0)$ be the probability of the RI onset occurrence when an initial state of the SDE \eqref{SDE_UVB2} is $x=(u,v,b)$. Then $p$ satisfies the following boundary value problem
\begin{equation}
\label{Hitting Probabilities}
\begin{cases}
    \mathcal{L}p(x)&=0 \qquad \text{if}\quad 0<v<\ell\\
    p(x)&=1 \qquad \text{if}\quad v=\ell\\
    p(x)&=0 \qquad \text{if}\quad v=0
\end{cases},
\end{equation}
where $\mathcal{L}$ is the operator \eqref{L1;Generator}.
\end{lemma}

\begin{proof}
The proof is standard and we give a sketch to illustrate the key idea.
Recall \eqref{tau+}-\eqref{tau0} and define $\tau:= \min\{\tau_+,\, \tau_0\}$ 
Clearly, $\{\tau_+<\tau_0\}=\{X_{\tau}=\ell\}$. Hence $p(x)=\P_x(X_\tau= \ell)$.
The random times $\tau_+$, $\tau_0$ and $\tau$ are stopping times with respect to the filtration  generated by $X$. 
Hence by the Dynkin's formula (see \cite[Chapter 2]{MR838085}),
\begin{equation}\label{Dynkin}
\E_xf(X_\tau)= f(x) + \E_x\Bigg[\int_0^\tau \mathcal{L}f(X_s) \,ds \Bigg]
\end{equation}
for all bounded functions $f$ in the domain of $\mathcal{L}$. From this and the fact that  \eqref{Hitting Probabilities} has unique solution, we can check that $p$ is the solution to  \eqref{Hitting Probabilities}.
\end{proof}

It should be noted that if the starting point $x_0=(u_0,v_0,b_0)$ is fixed (i.e., does not depend on $\epsilon$) and that $v_0>0$, then the probability of RI onset will tend to 1 as  $\epsilon\to 0$. This is because the ODE starting with $v_0>0$, which corresponds to the case $\epsilon=0$ (i.e., no random fluctuation), always hit level $\ell$ 
under the assumption $0<\ell<v_*$ (i.e. Assumption \ref{A:ell}. See also \refcite{KieuWang2017a}). By \cite[Lemma 5]{monter2010scaling} and Assumption 1, 
\begin{equation}\label{epsto0_1}
 \lim_{\epsilon\to 0} p(x_0)=1 \quad\text{for all }x_0=(u_0,v_0,b_0)\in \mathbb{R}\times (0,\ell]\times \mathbb{R}.
\end{equation}
This asymptotic behavior will be verified by our Monte-Carlo simulation to be presented in the next section (cf. Figure \ref{Histogram}), which shows indeed that the probability of RI onset increases to 1 when $\epsilon\to 0$ for each initial point  $x_0$.

\subsection{Distribution of RI onset time}
Assume that RI onset occurs, the next question one wishes to examine is how the RI onset time $\tau_+$ depends on the model initial conditions or parameters. 
For this, we can examine the Cumulative Distribution Functions (CDF) of $\tau_+$, conditioned on the occurrence of RI onset. 

Precisely, we let $F_{\text{on}}(t,x):=\mathbb{P}_x(\tau_+ \leq t\,|\,\tau_+<\tau_0)$ be the CDF of $\tau_+$, under the condition that RI onset occurs and the SDE \eqref{SDE_UVB2} starts from an initial value $x$. Let $p(x)=\mathbb{P}_x(\tau_+<\tau_0)$ be the probability of RI onset as in Lemma \ref{Lm:HittingProb}, then
\begin{equation}\label{Def:Fon}
F_{\text{on}}(t,x)=\frac{\mathbb{P}_x(\tau_+ \leq t\,,\,\tau_+<\tau_0)}{p(x)}.
\end{equation}
In Lemma \ref{L:Ponset2} below, we obtain the numerator of \eqref{Def:Fon}, and therefore $F_{on}$.
\begin{lemma}[Distribution of RI onset time]\label{L:Ponset2}
Let 
$G(t,x):=p(x)F_{\text{on}}(t,x)=\mathbb{P}_x(\tau_+ \leq t\,,\,\tau_+<\tau_0)$. Then $G$ satisfies the boundary value problem
\begin{align}
\frac{\partial G(t,x)}{\partial t}&= \mathcal{L}G(t,x) , \qquad t\in(0,\infty),\; x\in\{(u,v,b):\;0<v<\ell\} \label{PDE_Ha}\\
G(0,x)&=0 , \qquad\qquad \quad \qquad  \qquad \quad x\in \{(u,v,b):\;0<v<\ell\} \label{PDE_Hb}\\
G(t,x)&=0 , \qquad\qquad \quad t\in(0,\infty),\;x\in  \{(u,0,b),\,(u,\ell,b)\} \label{PDE_Hc},
\end{align}
where
$\mathcal{L}$ is the operator \eqref{L1;Generator}.
\end{lemma}

\begin{proof}
The initial condition \eqref{PDE_Hb} and the boundary condition \eqref{PDE_Hc} are clearly satisfied.
Let $H(t,x):=\mathbb{P}_x(\tau_+ > t\,,\,\tau_+<\tau_0)\,=\,p(x)-F_{\text{on}}(t,x)$.

Let  $X^{ab}$ be the absorbed diffusion \cite{bass1994probabilistic} obtained when $X$, the solution to  SDE \eqref{SDE_UVB2}, is absorbed upon hitting the boundary of the domain $D=\{(u,v,b):\;0<v<\ell\}$. 
By the strong Markov property of  $X^{ab}$,
\begin{align*}
H(t,x)=&\mathbb{P}_x(X^{ab}_t\in D,\,\tau_+<\tau_0)\\
=&\int_{D}\mathbb{P}_y(\tau_+<\tau_0)\,\mathbb{P}_x(X^{ab}_t\in dy)\\
=&\int_{D}p(y)\,p^{ab}(t,x,y)\,dy,
\end{align*}
where $p^{ab}(t,x,y)$ is the transition density of $X^{ab}$. By the backward Kolmogorov's inequality \cite[Chapter 5]{MR1121940}, we have 
$\frac{\partial H(t,x)}{\partial t}= \mathcal{L}H(t,x)$ and hence
\eqref{PDE_Ha}.
\end{proof}

Physically, Lemma \ref{L:Ponset2} informs us the probability of having an onset time no later than a time $t$ if the initial condition is $x=(u,v,b)$. Numerically, one can always solve  \eqref{PDE_Ha}-\eqref{PDE_Hc} to obtain $F_{\text{on}}(t,x)$. We can then compare the corresponding probability density function (i.e. its time-derivative $\frac{\partial F_{\text{on}}}{\partial t}$) with the histograms of RI onset statistics obtained from the numerical simulation of the MSD system \eqref{ODE_MSD} (cf. Figure \ref{Histogram}).

To obtain more quantitative insight about $\tau_+$, we establish in Theorem \ref{T:AsympDistri} below the limiting distribution of the onset time
probability density distribution for $\tau_{+}$ as $\epsilon\to 0$.
Let $(u(t), v(t), b(t))_{t\geq 0}$ be the solution of the ODE \eqref{ODE_MSD}  starting at $x_0=(u_0,v_0,b_0)$ and  $T=\inf\{t\geq0 \,; v(t)=\ell\}$, we then have
\begin{theorem}[Asymptotic distribution of RI onset time]\label{T:AsympDistri}
Suppose the initial state of the SDE \eqref{SDE_UVB2} is the same as that of  the ODE \eqref{ODE_MSD}; that is, $X_0=x_0=(u_0,v_0,b_0)$. Suppose  $v_0\in (0,\ell)$ and that the onset level $\ell$ satisfies Assumption \ref{A:ell}. Then as $\epsilon \to 0$, the  random variable $\epsilon^{-1}(\tau_+ - T)$ converges in distribution to the centered Gaussian random variable with variance
\begin{align}\label{Formula_Var2}
 \frac{\Sigma_{22}(T)}{\ell^2\,[u(T)+\ell]^2},
\end{align}
where   $\Sigma(t)=(\Sigma_{ij}(t))$ is the $3\times 3$  matrix 
\begin{align}\label{eq:Sigma}
\Sigma(t)= e^{\int_0^tA(r)\,dr} \cdot \left(\int_0^t e^{-\int_0^s A(r)\,dr} e^{-\int_0^s A^\intercal(r)\,dr}  \, ds\right) \cdot e^{\int_0^tA^\intercal(r)\,dr} 
\end{align}
and $A(t)=A_{x_0}(t)$ is the Jacobian matrix
\begin{align}\label{eq:matrixA}
    A(t)=D\mu(x(t))
    =
\begin{pmatrix}
 -v(t) & 2p\,v(t)-u(t)& -(p+1)\\
 -v(t) & -u(t)-2v(t)& 0\\
 b(t)+s & 1& u(t)-r
\end{pmatrix}.
\end{align}
\end{theorem}

An immediate consequence of Theorem \ref{T:AsympDistri} is an asymptotic formula for the variance of $\tau_+$, conditioned on RI onset occurrence ($\tau_+$ is infinity by convention if RI does not occur. So we should consider the conditional variance rather than the variance of $\tau_+$).
\begin{corollary}[Variance of RI onset time]\label{Cor:AsympDistri}
As $\epsilon \to 0$, the distribution of the RI onset time $\tau_+$ is well approximated by a Gaussian variable with mean $T$ and conditional variance
\begin{align}\label{Formula_Var}
     Var(\tau_+ \,|\,\tau_+<\tau_0) &\approx \epsilon^2 \, \frac{\Sigma_{22}(T)}{\ell^2\,[u(T)+\ell]^2}.
\end{align}
\end{corollary}

Corollary \ref{Cor:AsympDistri} is noteworthy because it captures the behavior of the conditional variance of $\tau_+$ in terms of the initial value $x_0$ as well as the model parameters $p,r,s$ as $\epsilon\to 0$, which is proportional to the variance of the additive noise $\epsilon^2$. Examination of its dependence on the model parameters $(p,r,s)$ shows that the probability distribution for $\tau_+$ is very close to a Gaussian distribution centered at $T$ when $\epsilon\to 0$, as will later be verified in our numerical simulations.

\begin{proof}[Proof of Theorem \ref{T:AsympDistri}]
Our proof is based on Theorem 1 in \refcite{monter2010scaling}
which gives an asymptotic result for a small noise stochastic diffusion equation
\begin{align}
dX_\epsilon(t)=&[\mu(X_\epsilon(t))+\epsilon^{\alpha_1}\Psi_\epsilon(X_\epsilon(t))]dt+ \epsilon\sigma(X_\epsilon(t))dW_t \label{BakhtinDiffusion}\\
X_\epsilon(0)=&x_0+\epsilon^{\alpha_2}\xi_\epsilon. \label{BakhtinDiffusion_initial}
\end{align}

We need to check the conditions of that theorem before we can apply it.
By taking $\xi_\epsilon \equiv 0$, $\Psi_\epsilon \equiv 0$, $\alpha_1=1$ and $\sigma(\cdot)=I_{3\times 3}$ the unit matrix  in \eqref{BakhtinDiffusion}- \eqref{BakhtinDiffusion_initial},  we obtain  \eqref{SDE_UVB}, with the unperturbed initial condition $X_0=x_0$. We let $M$ be the hyperplane $M=\{(u,v,b)\in \mathbb{R}^3:\, v=\ell \}$ in $\mathbb{R}^3$. Then the hitting time $\tau_\epsilon$ in   Theorem 1 of \cite{monter2010scaling} is exactly the RI onset time $\tau_+$ defined in \eqref{tau+}.


\noindent
{\bf Step 1: Joint convergence. }
Recall that the deterministic time $T$ defined by \eqref{Def:T_v} is the first time the trajectory $v$ of the ODE \eqref{Eg1} hits level $\ell$. Denote by $z:=(z_1,z_2,z_3)=(u(T), \ell, b(T))$ the point where the trajectory of the MSD system \eqref{ODE_MSD} hits  $M$. In \refcite{bakhtin2017scaling} it is assumed that the deterministic vector field $\mu$ is smooth, and the deterministic time defined by \eqref{Def:T_v} must satisfy $0<T<\infty$. Moreover, it is assumed that $\mu(z)$ does not belong to the tangent space $T_zM$ of  $M$ at the point $z$ (or in other words the orbit of the system \eqref{ODE_MSD} intersect $M$ and the crossing is transversal). These  assumptions are satisfied by our MSD system \eqref{ODE_MSD} under Assumption \ref{A:ell}.
Therefore, we can indeed apply  Theorem 1 of \refcite{monter2010scaling}. 

Let $\pi_\mu$ be the  projection onto $span(\mu(z))$ along $T_zM$ and $\pi_M$  the  projection onto $T_zM$ along $span(\mu(z))$; see Figure \ref{Fig:Proj}. Then for any vector $\eta\in \mathbb{R}^3$,  $\pi_\mu \eta \in \mathbb{R}$ and $\pi_M \eta \in T_zM$ satisfy
 \[\eta=\pi_\mu \eta \cdot \mu(z) + \pi_M \eta.\]

\begin{figure}[h!]
\centering
\includegraphics[scale=0.4]{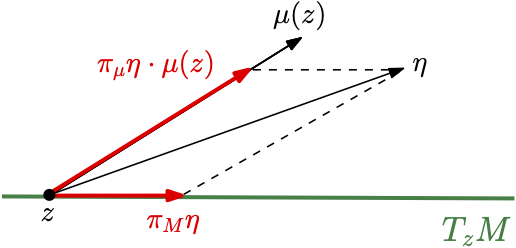}
\caption{Illustration of the projections $\pi_\mu \eta \in \mathbb{R}$ and $\pi_M \eta \in T_zM$. Given  vectors $\eta,\,\mu(z)\in \mathbb{R}^3$ and the tangent plane $T_zM$, we have
$\eta=\pi_\mu \eta \cdot \mu(z) + \pi_M \eta$.
}
\label{Fig:Proj}
\end{figure}
 
Note that in our case, the tangent space $T_zM$ is exactly $M$ itself since it is a plane. 
 Theorem 1 of \refcite{monter2010scaling} asserts the following convergence in distribution as $\epsilon \to 0$.
\begin{equation}\label{SpecializedBakhtinResult}
    \epsilon^{-1}\left(\tau_+-T,\,X_\epsilon(\tau_+)-z\right) \xrightarrow{d} \left(-\pi_\mu \phi_0(T),\pi_M\phi_0(T)\right)
\end{equation}
where
\begin{equation}\label{randomvectorphi}
    \phi_0(t)= \Phi_{x_0}(t) \int_0^t\Phi_{x_0}(s)^{-1} \, dW(s)
\end{equation} 
is a random vector in $\mathbb{R}^3$. 
The matrix-valued function $\Phi_{x_0}(t)=e^{\int_0^tA_{x_0}(r)\,dr}$ solves the equation
\begin{align*}
    \frac{d}{dt}\Phi_{x_0}(t)&=
     A_{x_0}(t)
    \Phi_{x_0}(t)\\
    \Phi_{x_0}(0)&= I_{3\times 3},
\end{align*}
where $A_{x_0}(t)$ is given by \eqref{eq:matrixA}.

\noindent
{\bf Step 2: Projection and variance computation. }
The  above 3-dimensional random vector \eqref{randomvectorphi} is  Gaussian distributed with mean zero and co-variance matrix $\Sigma(t)$ given by \eqref{eq:Sigma}. That is,
$
\phi_0(t)\sim \mathcal{N}(0, \Sigma(t))
$.
Let $\phi_0(T)=(\phi_{0,1}(T),\phi_{0,2}(T),\phi_{0,3}(T))$.
Clearly, $\phi_{0,2}(T) \sim \mathcal{N}(0,\Sigma_{22}(T))$. 

By definition of the projections, we have 
$$\phi_0(T)=\pi_\mu \phi_0(T)  \mu(z) + \pi_M \phi_0(T),$$  
where $\pi_\mu \phi_0(T) \in \mathbb{R}$ and $\pi_M \phi_0(T) \in T_zM$.  See Figure \ref{Fig:Proj} with $\eta=\phi_0(T)$ for an illustration.
Since $T_zM$ is parallel to the $(u,b)$-plane, the second coordinate (i.e. the $v$-coordinate) of $\phi_0(T)$ is the same as that of $\pi_\mu \phi_0(T)  \mu(z)$. That is,
\[
\phi_{0,2}(T)=\pi_\mu \phi_0(T) \cdot \mu_2(z).
\]

This implies that
\begin{equation}\label{ProjProof}
     \pi_\mu \phi_0(T)= \frac{\phi_{0,2}(T)}{\mu_2(z)}
\end{equation}
and
so $\pi_\mu \phi_0(T)$ is a centered Gaussian vector with variance
$\frac{\Sigma_{22}(T)}{\mu_2(z)^2}$.

\noindent
{\bf Step 3: Conclusion. }
In conclusion, from \eqref{SpecializedBakhtinResult}, for $\epsilon$ close to zero we get that in distribution, 
\begin{equation}\label{AsympDistri}
\big(\tau_{+}, X_\epsilon(\tau_+)\big) \approx (T,z) + \epsilon \big(-\pi_\mu\phi_0(T), \pi_M \phi_0(T)  \big).
\end{equation}


\noindent
From \eqref{AsympDistri} and \eqref{epsto0_1}, the conditional expectation
\begin{align}\label{CondEx}
    \E_{x_0}[\tau_+ \,|\,\tau_+<\tau_0] \to T \; \mbox{as} \; \epsilon \rightarrow 0,
\end{align}
because $\phi_0(T)$ is a centered Gaussian vector. For the conditional variance,
\begin{align*}
    Var(\tau_+ \,|\,\tau_+<\tau_0)\approx 
    \epsilon^2 \frac{\Sigma_{22}(T)}{\mu_2(z)^2}= \epsilon^2 \frac{\Sigma_{22}(T)}{z_2^2(z_1+z_2)^2}.
\end{align*}
\end{proof}

\subsection{Hitting analysis in one-dimension for small initial values}\label{S:1-dim}
The behavior of TC dynamics as shown in Figure \ref{fig1} reveals an important characteristic of the stochastic forcing in TC development. Specifically, the pre-RI period before TC intensity rapidly amplifies is characterized by very slow evolution, much like a constant-forcing dynamical system. One can therefore exploit further the consequence of this property to study RI onset by considering a general one-dimensional SDE model for the $v$ component, which can provide more insights into the variability of RI onset time. Specifically, we wish to examine herein a particular case in which \textit{the noise $\epsilon$ is fixed and $v_0$ is small} ($v_0\to 0$). This case differs from \eqref{epsto0_1} and Theorem \ref{T:AsympDistri} which focus on the probability of RI onset and the distribution of the onset time for a limit of the small noise $\epsilon\to 0$ with a fixed initial condition $(u_0,v_0,b_0)$. As such, the behaviors of RI onset for a fixed noise $\epsilon$ but small $v_0$ are not unclear from Theorem \ref{T:AsympDistri}, which we wish to examine further in this subsection. 

For this purpose, we observe from our Monte-Carlo simulations of the MSD system to be presented in Section 4 that (see the lower left panels in Figure \ref{Histogram})  
\begin{enumerate}
    \item[$\mathcal{O}$1:] the probability of RI onset gets smaller as $v_0\to 0$.
    \item[$\mathcal{O}$2:] the conditional distribution of the RI onset time $\tau_+$ (given that an RI onset occurred) is skewed to the left and has a \textit{smaller} averaged value than the deterministic onset time $T$, and
\end{enumerate}
This regime (i.e., $v_0$ is very small compared with the noise) is challenging to analyse, even for the 3-dimensional SDE like \eqref{SDE_UVB2}, because a standard Gaussian approximation is no longer valid. However, it is possible to offer some insight into the aforementioned observations through the following general 1-dimensional SDE:
\begin{equation}\label{1D}
dZ_t=F(Z_t)\,dt + \epsilon dW_t,
\end{equation}
where $W$ is the Wiener process in $\R$, and $F:\,\R\to\R_+$ is an arbitrary given smooth function such that
\begin{equation}\label{A:F}
F(0)=0 \quad\text{and}\quad F(x) >0 \text{ for }x>0.
\end{equation}
Our aim here is to compare the qualitative behavior of the  $V$-component of \eqref{SDE_UVB2} and the process $Z$ solving \eqref{1D},  when the initial value $v_0$ is ``small". How ``small" the initial value is depends on the  fixed noise level $\epsilon$, as quantified in Theorem \ref{T:asymP} below.

\begin{remark}[Why consider 1-dim SDE ?]\rm
We should emphasise here that it is not our intention to directly applying \eqref{1D} to the MSD system. Instead, the aim of this 1D system is to qualitatively capture the behaviors of TC dynamics during the initial development up to the RI onset moment for which the MSD forcing can be approximated as a constant. The advantage of this analysis lies in the fact that the forcing function $F$ can be quite general; that is, we do not require any specific functional form for $F$ and so our analysis for the 1D system \ref{1D} works for a larger class of forcing functions $F$. 
\end{remark}

We consider the hitting times at the endpoints 0 and  $\ell$. That is, $$\tau_i:=\inf\{t\geq 0 :\;  Z_t=i\} \quad \text{for } i=0,\,\ell.$$
Analogous to the probability of RI onset is
$\P_x(\tau_{\ell}<\tau_0)$,  the probability of hitting $\ell$ before 0 provided that \eqref{1D} starts at $Z_0=x$. 
Lemma \ref{Lm:HittingProb_1d} below gives an exact formula for this, which is not available in higher dimensions in general.
\begin{lemma}\label{Lm:HittingProb_1d}
The probability of hitting $\ell$ before 0 for $Z$ in \eqref{1D} starting at $x\in[0,\ell]$ is
\begin{equation*}
\P_x(\tau_{\ell}<\tau_0)=\frac{\int_{0}^x k_\epsilon(y)\,dy}{\int_{0}^\ell k_\epsilon(y)\,dy},
\end{equation*}
where $k_\epsilon$ is the function
\begin{equation}\label{Def:kepsilon}
    k_\epsilon(y)=\exp\left\{\frac{-2}{\epsilon^2}\int_0^yF(t)dt\right\}.
\end{equation}
\end{lemma}

The following result quantifies a dichotomy for the probability $\P_{\epsilon^{\alpha}}(\tau_{\ell}<\tau_0)$ 
 of hitting $\ell$ before 0, starting at $\epsilon^{\alpha}$. Namely this probability is close to 0 if the starting point is small ($\alpha$ large), and close to 1 when the starting point is large  ($\alpha$ small).


\begin{theorem}[Asymptotic hitting probability]\label{T:asymP}
Suppose $F$ is a smooth function satisfying \eqref{A:F} and $F'(0) >0$.
For all $c>0$, the probability of hitting $\ell$ before 0 for $Z$ in \eqref{1D} starting at $c\, \epsilon^{\alpha}$ satisfies
\[ \lim_{\epsilon \rightarrow 0} 
\P_{c\,\epsilon^{\alpha}}(\tau_{\ell}<\tau_0)
=
\begin{cases}
1&\text{if}\ \alpha\in (0,1)  \quad \text{i.e. starting point is not small}\\
\erf(c \sqrt{F'(0)})&\text{if}\ \alpha=1\\
0 &\text{if}\ \alpha>1 \qquad \text{i.e. starting point is very small},
\end{cases} 
\]
where  $\erf(x):=\frac{2}{\sqrt{\pi}}\int_0^x e^{-z^2} \,dz$ is the error function.
\end{theorem}

\begin{remark}\label{Rk:asymP}
Analogous to the \textit{RI onset indicator} \eqref{Def:Indicator} is the inverse $h^{-1}_{\epsilon}(0,8)$. From the critical case $\alpha=1$, $\lim_{\epsilon\to 0} h(c\,\epsilon)=\erf(c \sqrt{F'(0)})$. Then
\[
h^{-1}_{\epsilon}(0.8) \approx \frac{\epsilon}{\sqrt{F'(0)}}\,\phi^{-1}(0.8)
\]
is linear in $\epsilon$ when $\epsilon\approx 0$. This is consistent with the approximately linear curve in Figure \ref{Probability2}.
\end{remark}

Now we condition on the event $\{\tau_{\ell}<\tau_0\}$ and consider the conditional distribution of the hitting time $\tau_{\ell}$. We shall compute conditional expected time 
This is analogous to conditioning on RI onset occurrence and consider the conditional distribution of the RI onset time. 
Precisely, we shall compute the conditional expected time
\begin{equation}\label{E:tau and Ttau}
\E_{x}[\tau_{\ell}\,|\,\tau_{\ell}<\tau_0].
\end{equation}

For the rest of this section, we obtain an explicit formula for this conditional expected time  in Lemma \ref{Lem:ConEXP 1-d}, and study its asymptotic behavior in Theorem \ref{T:AsymTime} below.
\begin{lemma}\label{Lem:ConEXP 1-d}
For  $x \in (0,\ell]$,
\begin{equation}\label{E:ConEXP 1-d}
\E_x[\tau\,|\,\tau_\ell<\tau_0]=\frac{2}{\epsilon^2} \frac{1}{\int_0^x k_\epsilon(z) dz} 
   \left\{ \int_x^\ell \int_0^x k_\epsilon(z) k_\epsilon(u) \int_u^z  \frac{p(y)}{k_\epsilon(y)}\, dy \,du \,dz\right\},
\end{equation}
where $p(x)=\P_x(\tau_\ell<\tau_0)$ is the probability in Lemma \ref{Lm:HittingProb_1d}.
\end{lemma}

\medskip




Theorem \ref{T:AsymTime} below asserts that as the starting point $x\to 0$, 
\begin{equation}\label{E:AsymTime2b}
\E_{x}[\tau_{\ell}\,|\,\tau_{\ell}<\tau_0] \,\approx\,  \Psi(\epsilon) - \frac{x^2}{3\epsilon^2}
\end{equation} for some positive number $\Psi(\epsilon)$.
\begin{theorem}[Asymptotic conditional hitting time]\label{T:AsymTime}
Suppose $F$ is a continuous function. For each fixed noise level $\epsilon>0$, 
\begin{equation}\label{E:AsymTime1}
\lim_{x\to 0}\E_{x}[\tau_{\ell}\,|\,\tau_{\ell}<\tau_0]= \Psi(\epsilon)
\end{equation}
where
\begin{align}
\Psi(\epsilon)= \frac{2}{\epsilon^2} \int_0^\ell k_\epsilon(u) \int_0^u \frac{p(y)}{k_\epsilon(y)}\, dy \,du \in (0,\infty). \label{E:AsymTime_Psi}
\end{align}
Furthermore, suppose $F$ satisfies \eqref{A:F} and $F'(0)>0$. Then
\begin{equation}\label{E:AsymTime2}
\lim_{x\to 0}\frac{\Psi(\epsilon)-\E_{x}[\tau_{\ell}\,|\,\tau_{\ell}<\tau_0]}{x^2} = \frac{1}{3\epsilon^2}.
\end{equation}
\end{theorem}

\medskip

Theorem \ref{T:AsymTime} implies that, fixing a noise level $\epsilon>0$,  the conditional expected hitting time $\E_{x}[\tau_{\ell}\,|\,\tau_{\ell}<\tau_0]$ stays bounded as $x\to 0$. This is in contrast to the deterministic analogue  (which tends to infinity in the order of $O(-\log x)$ as $x\to 0$ when $F'(0)>0$). Since
\begin{equation}
  \Psi(\epsilon) - \frac{x^2}{3\epsilon^2} \,\ll\, O(-\log x)\quad \text{as }x\to 0,
\end{equation}
Theorem \ref{T:AsymTime} provides a possible  explanation to the observation that  the  conditional expected hitting time is \textit{shorter} than the deterministic hitting time, mentioned in observation $\mathcal{O}$2 at the beginning of this section.
%
%
\section{Numerical results for the stochastic MSD system \eqref{SDE_UVB2}}
\subsection{Algorithm}
From the practical standpoint, Lemma \ref{Lm:HittingProb}, Lemma \ref{L:Ponset2} and Theorem \ref{T:AsympDistri}  presented in Section \ref{S:Theoretical} are useful for RI forecast applications, because they directly indicate the probability of RI onset occurrence as well as the variability of RI onset time. In this section, we will present numerical investigation to validate a number of theoretical results presented in Section 3, from which further examination of RI onset on various model parameters and initial conditions can be obtained. In particular, we wish to verify the variance formula \eqref{eq:Sigma} in Corollary \ref{Cor:AsympDistri}  for RI onset time because of its importance in practical applications. While this formal variance expression is mathematically significant, its direct calculation is challenging because of the matrix exponent and integration that are very sensitive to matrix operations. As such, we present in this section a numerical algorithm to compute the matrix $\Sigma(t)$ efficiently.

For the numerical purposes, we observe that the variance matrix $\Sigma(t)$, defined in \eqref{eq:Sigma}, solves the following differential equation:
\begin{align}\label{SigmaEqt}
    \frac{d \Sigma(t)}{dt} = I_{3\times 3}+A_{x_0}(t)\Sigma(t)+ \Sigma(t)A_{x_0}(t)^\intercal,
\end{align}
where the matrix $A_{x_0}(t)$ is defined in \eqref{eq:matrixA}. The above Eq. \eqref{SigmaEqt} can be indeed derived by rewriting
\begin{align}
\Sigma(t)=&\Phi_{x_0}(t) \cdot N(t) \cdot (\Phi_{x_0}(t))^\intercal, \label{Sigmaproduct}
\end{align}
where $N(t)=\left(\int_0^t \Phi_{x_0}(s)^{-1} (\Phi_{x_0}(s)^{-1})^\intercal \, ds\right)$, and $\Phi_{x_0}(t)$ is the solution of the following differential equation \cite{monter2010scaling},
\begin{align*}
    \frac{d}{dt}\Phi_{x_0}(t)&=
     A_{x_0}(t)
    \Phi_{x_0}(t)
    \label{Linearization}\\
    \Phi_{x_0}(0)&= I_{3\times 3}.
\end{align*}
After taking derivatives in \eqref{Sigmaproduct} and applying the product rule, we get,
\begin{align}
    \frac{d}{dt}\Sigma(t)=&\frac{d}{dt}\Phi_{x_0}(t) \cdot N(t)\cdot (\Phi_{x_0}(t))^\intercal +
    \Phi_{x_0}(t) \cdot \frac{d}{dt}N(t)\cdot (\Phi_{x_0}(t))^\intercal \notag\\
    &+\Phi_{x_0}(t) \cdot N(t)\cdot \frac{d}{dt}(\Phi_{x_0}(t))^\intercal.
\end{align}
By \eqref{Linearization} and the fact that $\frac{d}{dt}N(t)= \Phi_{x_0}(t)^{-1} (\Phi_{x_0}(t)^{-1})^\intercal$, we thus have
\begin{align}\label{eq:sigmafull}
    \frac{d}{dt}\Sigma&=A_{x_0}\Phi_{x_0} \, N\, \Phi_{x_0}^\intercal +
    \Phi_{x_0} \Phi_{x_0}^{-1} (\Phi_{x_0}^{-1})^\intercal \, \Phi_{x_0}^\intercal+
  \Phi_{x_0} \, N \, \Phi_{x_0}^\intercal A_{x_0}^\intercal.
\end{align}
Using \eqref{Sigmaproduct} again and rearranging the right hand side of Eq. \eqref{eq:sigmafull}, we thus obtain Eq. \eqref{SigmaEqt} for the variant matrix $\Sigma(t)$.

The particular benefit of this differential equation approach for $\Sigma(t)$ instead of the formula \eqref{eq:Sigma} in Corollary \ref{Cor:AsympDistri}  is that it allows for integrating the matrix equation \eqref{SigmaEqt} forwards in time from any initial condition up to any given time $t$ without the need of explicitly computing the exponent of matrix integration in Eq. \eqref{eq:Sigma}. Note however that this algorithm requires computing the coefficient matrix $A(t)$ along the trajectory, which is the Jacobian matrix of the model state as seen in Eq. \eqref{eq:matrixA}. As a result, we have to integrate the deterministic model \eqref{MSD} first and store the entire trajectory $(u(t),v(t),b(t))$ before the integration of \eqref{SigmaEqt} can be carried out. 

Along with the above numerical algorithm to obtain the variance formula in Corollary \ref{Cor:AsympDistri} , Monte-Carlo simulations of the MSD model \eqref{SDE_UVB2} will be also carried out to verify Corollary \ref{Cor:AsympDistri}. For these Monte-Carlo simulations, the MSD system \eqref{SDE_UVB2} is integrated by using the Runge-Kutta fourth order scheme with time step $dt=0.001$. As mentioned in \refcite{NguyenChanhFan}, the stochastic forcing in the MSD system \eqref{SDE_UVB2} is additive with no state dependence. Thus, the Runger-Kutta scheme can be applied to the deterministic part of Eq. \eqref{SDE_UVB2}, with the stochastic forcing added at each time step. This method retains the fourth order accuracy for the deterministic part, while the stochastic accuracy order first order as for the Euler–Maruyama scheme \cite{KloedenPlaten1992}. 

Because of the random nature of stochastic forcing, all Monte-Carlo simulations in this study are carried out with 1000 realizations for each choice of initial conditions and random forcing amplitude $\epsilon$. A fixed set of parameters for the MSD model with $(p,r,s)=(200,0.25,0.1)$ similar to those used in \refcite{NguyenChanhFan} is also employed in all simulations. These parameters are typical for TCs in real atmospheric conditions as shown in \refcite{Kieu2015} and \refcite{NguyenChanhFan}. By comparing the results from the numerical integration of Eq. \eqref{SigmaEqt} and the Monte-Carlo simulations of the MSD system, the validity of the theoretical results in the previous section can be assessed. 
%
%
\subsection{RI onset probability}
We investigate first in this subsection the probability of RI onset occurrence as presented in Lemma \ref{Lm:HittingProb}, using the Monte-Carlo simulations of the MSD system. Figure \ref{Probability1}a shows the probability on RI onset $p(u_0,v_0,b_0)$ as a function of the initial condition $v_0$. Consistent with observations \cite{Fischer_etal18}, one notices that the RI occurrence probability quickly increases with $v_0$, regardless of the random forcing amplitude $\epsilon$. For $\epsilon<10^{-2}$, the RI occurrence probability reaches the value of $\sim 1$ for all $v_0>0.05$. This means RI will be almost guaranteed to occur, because a sufficiently strong initial vortex would practically mean that a TC is well organized and so it will most likely undergo RI. 
%
%
\begin{figure}[tbh]
\centering
\includegraphics[scale=0.38]{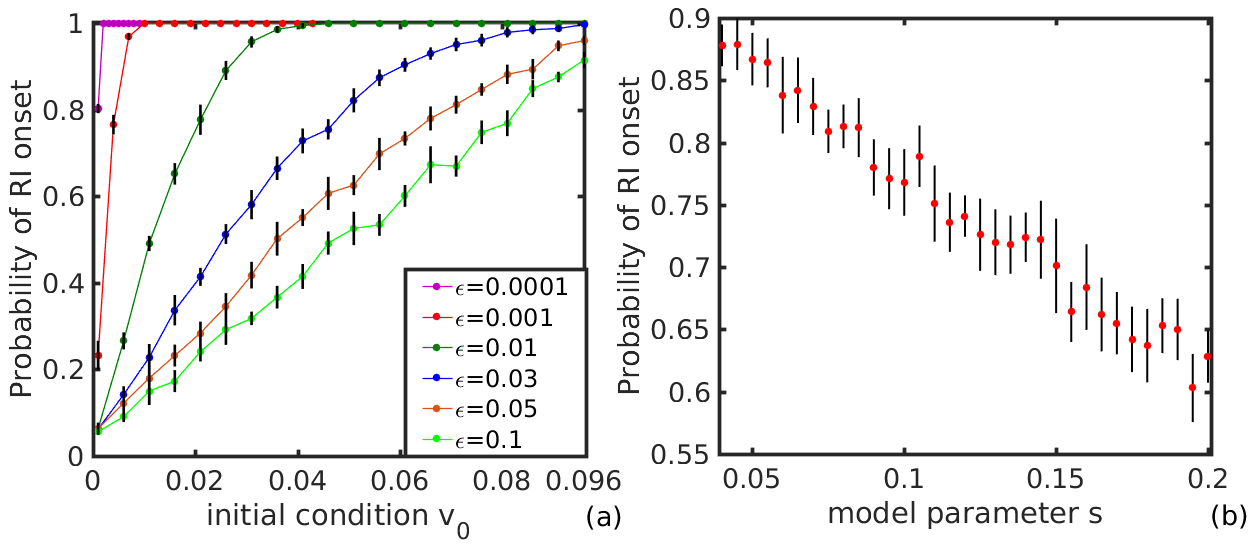}
\caption{(a) Probability of RI onset as a function of the initial wind component $v_0$, with different values of $\epsilon$ and $s=0.1$; and (b) probability of RI onset as a function of  $s$, where $v_0=0.02$ and $\epsilon=0.01$. In both figures, the remaining parameters are $(p,r)=(200,0.25)$ and $(u_0,\, b_0)=(-0.01,\,0.0001)$. 10 experiments each with 100 realizations have been done at each parameter value to achieve the error bars with 95 percent confidence interval.}
\label{Probability1}
\end{figure}

As the random fluctuation increases ($\epsilon > 0.03$), one noticed however that the probability for RI occurrence increases slower and approaches 1 only when $v_0$ is sufficiently large ($>$ 0.1). This threshold justifies the hereinafter use of $\ell=0.1$ for RI onset time in the MSD system (this level 0.1 for $v_0$ in the non-dimensional unit corresponds to $\sim 10 m s^{-1}$ in full physical dimension, which is consistent with previous idealized studies of RI onset. See e.g., \refcite{Kieu_etal2013,Kieu_etal2021}).        

Given the strong dependence of TC development on ambient environment, it is thus natural to expect that RI onset probability should be governed by not only initial conditions but also environmental factors. Among the three model parameters $(p,r,s)$, we note that $s$ is most sensitive to ambient environment because it represents the stratification of the troposphere \cite{Kieu2015,KieuWang2017a}. Thus, Figure \ref{Probability1}b shows the dependence of RI probability as a function of $s$ with fixed values for $\epsilon=0.01$, $v_0=0.02$ and all other parameters. Consistent with the previous studies on weaker intensity for more stable troposphere \cite{Shen_etal2000,HillLackmann2011,Tuleya_etal2016,MoonKieu2017,Ferrara_etal2017,KieuZhang2018,DownsKieu2020}, one notices in Figure \ref{Probability1}b that RI onset probability decreases quickly as $s$ is larger (i.e., the troposphere becomes more stable). Given the same initial vortex strength, an increase of $s$ from 0.1 to 0.2 could reduce the RI onset probability from 80 to 60\%, which is substantial in operational forecast. For smaller values of $v_0$, this drop in RI onset probability is even much faster. The implication of this result is significant, as it suggests that the environmental static stability is a key parameter not only for the TC maximum intensity, but also for RI onset prediction.

A different way to examine the sensitivity of RI onset probability in operational practice is to determine what value of the initial TC strength $v_0$ would allow for at least, e.g., 80\% RI probability as a function of the random magnitude $\epsilon$. This 80\% threshold is generally sufficient for most practical purposes to ensure that RI onset will be very likely to occur, from which timely risk management can be prepared. 

In this regard, Figure \ref{Probability2} shows the minimum initial TC strength $I^{\epsilon}_0$ to meet the 80\% RI onset probability threshold as a function of $\epsilon$. 
Here, we define $I^{\epsilon}_0$, which can be considered as an \textbf{RI onset indicator}, as the unique number within $(0,\ell)$ such that
\begin{equation}\label{Def:Indicator}
p(u_0,\,I^{\epsilon}_0,\,b_0) =0.8.
\end{equation}
%
%
\begin{figure}[tbh]
\centering
\includegraphics[scale=0.4]{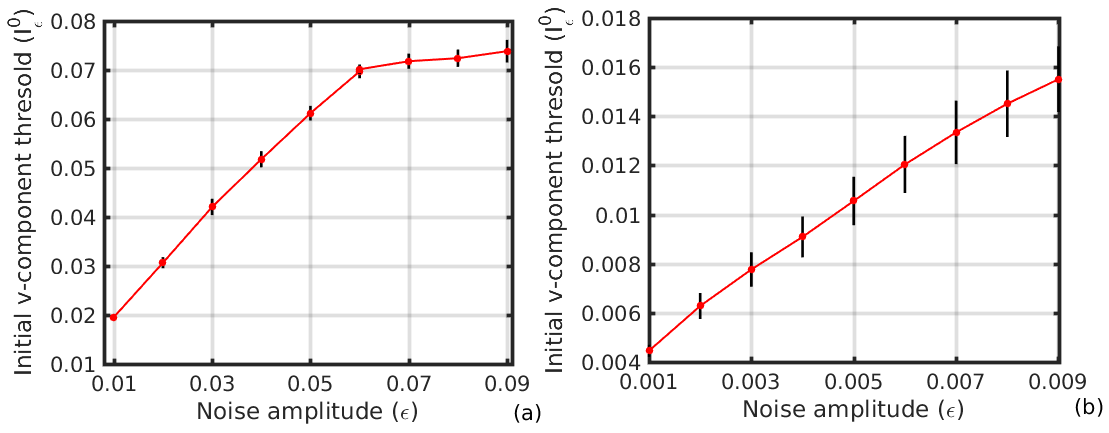}
\caption{Dependence of the smallest value of $v_0$ at which the probability of RI onset reaches a $0.8$ level, denoted as $I^{\epsilon}_0$ defined in \eqref{Def:Indicator}, on different ranges of the noise level $\epsilon$ including (a) $\epsilon \in [0-0.1]$, and (b) a zoom in for $\epsilon \in [0.001-0.01]$. Other parameter settings include $(p,r,s)=(200, 0.25, 0.1)$ and  $u_0=-0.01,\,b_0=0.0001$.}
\label{Probability2}
\end{figure}

Consistent with our theoretical results, $I^{\epsilon}_0$ increases linearly with $\epsilon$ when $\epsilon$ is small. It also appears that $I^{\epsilon}_0$ levels off for $\epsilon > 0.07$. The limit of a small $\epsilon$ is of interest, as it reveals that the MSD system behaves similarly to an one-dimensional stochastic system with autonomous forcing as presented in Section 3.3. As explained in Remark \ref{Rk:asymP}, Theorem \ref{T:asymP} shows that the linear dependence of $I^{\epsilon}_0$ on $\epsilon$ is always valid for a very general one-dimensional system, so long as the forcing does not vary much prior to RI onset. This is applied well to the MSD system as seen, e.g., in Figure \ref{fig1}, which shows that TCs evolves very slowly during the pre-RI onset period. Physically, this result confirms that a larger random noise would require a stronger initial intensity so that RI onset can be more likely to occur. Note that when the initial intensity is sufficiently large, the random noise will have less of an impact because the RI onset will almost guarantee to occur (at a 80\% level) for those initially strong intensity states. 
%
%
\subsection{RI onset timing variability}
Given the probability of RI onset occurrence as presented in the previous section, we wish to verify next the distribution of RI onset time as given by Theorem \ref{T:AsympDistri} and related corollary \ref{Cor:AsympDistri}. Because RI onset is almost guaranteed to occur when $v_0$ is sufficiently large as shown in Figure \ref{Probability2}, we will consider a specific case in which the hitting level $\ell$ for RI onset (i.e. $v$ component) is $\ell=0.1$.

Similar to the previous section, our main focus herein will be again on how the distribution of $\tau_+$ changes with the initial condition for the $v$ component (i.e., $v_0$), while keeping the other two components ($u_0,b_0$) at the same values of $u_0=-10^{-2}, b_0=10^{-4}$. This is because $v_0$ practically represents the intensity of a TC vortex during its initial stage of development. During this tropical disturbance stage, there is no strong dynamical constraint among the scales of TCs and one can therefore assign relatively independent values for $u_0,v_0,b_0$. As the tropical disturbance grows, its dynamics will be however governed by the TC scale dynamics and they cannot evolve independently. 

To have a broad picture of the variability of RI onset time, Figure \ref{Histogram} shows the histograms of $\tau_{+}$ for a range of $v_0$ and $\epsilon$. Here, these histograms are constructed from 1000 Monte-Carlo simulations, using the default values for the parameters and initial conditions as mentioned in Section 4.1. One notices in Figure \ref{Histogram} an expected behavior of the $\tau_+$ variability, with a narrower distribution of $\tau_{+}$ for smaller $\epsilon$ when $v_0 \geq 0.01$. That is, a smaller random forcing would result in less variability in RI onset timing, which is consistent with real TC development. 

Of further interest from Figure \ref{Histogram} is that for each fixed initial condition $v_0$ (i.e. for each row), the conditional distribution of $\tau_+$ gets closer to a probability density function centered around the deterministic onset time $T$ defined in \eqref{Def:T_v} as $\epsilon\to 0$. This indicates that the deterministic RI onset forecast will be more reliable for either smaller stochastic noises or stronger initial intensity. 
%
%
\begin{figure}[tbh]
\centering
\includegraphics[scale=0.36]{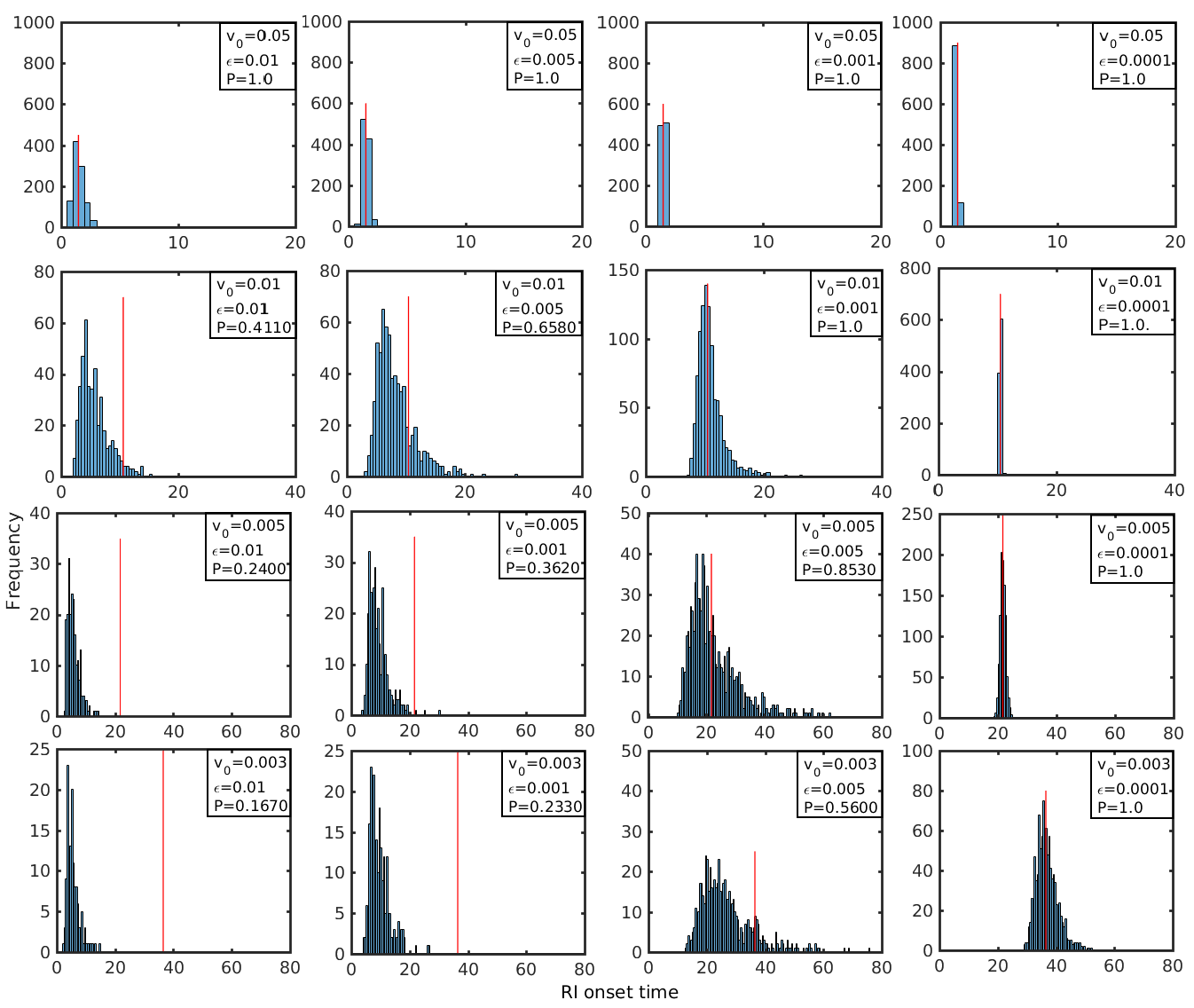}
\caption{Histograms of RI onset time $\tau_+$ defined in  \eqref{tau+} for various values of initial conditions $v_0$ and the noise amplitude $\epsilon$, conditioned on the event $\{\tau_+<\tau_0\}$. Note that for each 1000 realizations of the stochastic system \eqref{SDE_UVB2}, only a fraction $P$ of them hit $\{v=0.1\}$ before hitting $\{v=0\}$ and so only these trajectories are counted (these probability values $P$ are given in the upper right boxes). The red vertical line shows the time $T$,  defined in \eqref{Def:T_v}, obtained from the deterministic MSD system as $v(t)$ hits level $v=0.1$. In all histograms, the parameters are $p=200, r=0.25, s=0.1$, the initial values are $u_0=-0.01, b_0=0.0001$ are used.}
\label{Histogram}
\end{figure}

For very small values of $v_0$ (i.e., weaker initial intensity), the center of the $\tau_+$ distribution is shifted farther away from the deterministic time $T$ as $\epsilon$ increases (see the lower left panels in Figure \ref{Histogram}. This is because random fluctuations, which are proportional to $\epsilon$, are now much larger than the initial condition that TC development is no longer determined by $v_0$. Instead, the variability of $\tau_+$ is more a result of $\epsilon$ alone. So long as $v_0 \ll \epsilon \ll 1$, the TC initial condition becomes irrelevant to RI onset. This characteristic of RI onset timing uncertainty is also consistent with the probability of RI onset occurrence shown Figure \ref{Histogram} (see the RI onset probability $P$ in the upper right boxes). 

From the mathematical perspective, the above behavior of the MSD system for the limit of small $v_0$ can be understood by again using the a general one-dimensional SDE model presented in Section 3.3. So long as TC dynamics evolves slowly prior to RI onset, one can obtain an exact dependence of the center of $\tau_+$ histogram on $\epsilon$ in terms of the stochastic conditioned diffusion process (see Lemma 2). That is, the random noise in the MSD system induces a modified drift along the gradient of probability density, which results in a faster approach to the $\ell$ level as shown in Figure \ref{Histogram}. Thus, a smaller value of $v_0$ indicates less likely for RI to occur. For $\epsilon$ sufficiently larger than $v_0$, the probability $P$ for RI onset occurrence is quickly reduced below $50\%$, regardless of value of $v_0$ (see lower left panels in Figure \ref{Histogram}).   

To facilitate our comparison of the above results obtained from the Monte-Carlo simulations with that from Theorem \ref{T:AsympDistri} , the dependence of the variance of RI onset time on $v_0$ for each value of $\epsilon$ is summarized in Figure \ref{VarianceNumerics}. Consistent with that shown in Figure \ref{Histogram}, the variance of $\tau_+$ decreases for larger $v_0$ all range of $\epsilon$ as expected. Note that the conditional variance of $\tau_+$ (conditioned on occurrence of RI onset) also increases as $\epsilon$ increases, suggesting that the variability of the RI onset becomes larger when the amplitude of random forcing increases.
%
%
\begin{figure}[tbh]
\centering
\includegraphics[scale=0.35]{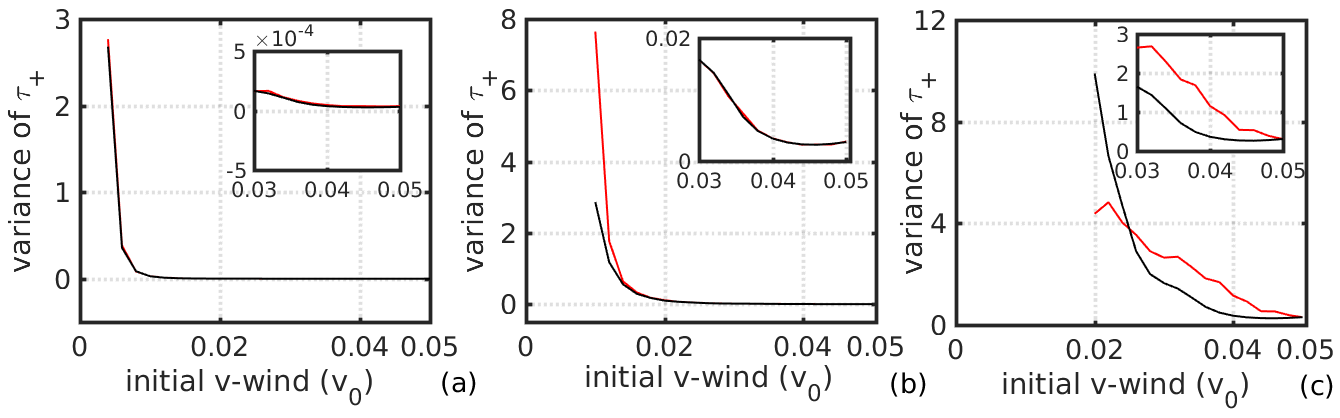}
 \caption{(a) Diagram shows the  variance of the RI onset time $\tau_+$ conditioned on RI onset occurrence  as a function of $v_0$ for noise $\epsilon=10^{-4}$, which is obtained from Corollary \ref{Cor:AsympDistri}  (black) and from Monte-Carlo simulation (red). (b)-(c) Similar to (a) but for $\epsilon=10^{-3}$, and $\epsilon=10^{-2}$. Upper right corner panels show the zoom in of these $Var(\tau_+|\tau_+<\tau_0)$ values for $v_0 \in [0.03-0.05]$.}
\label{VarianceNumerics}
\end{figure}

Comparing to the conditional variance of $\tau_+$ obtained from Corollary \ref{Cor:AsympDistri}  using the numerical integration of Eq. \eqref{eq:sigmafull} (see the black curves in Figure \ref{VarianceNumerics}), it is evident that Corollary \ref{Cor:AsympDistri}  captures  consistent characteristics of the conditional variance of $\tau_+$  as a function of $\epsilon$. This is especially true when $\epsilon$ is much smaller than $v_0$ (Figure \ref{VarianceNumerics}a-b), which show a good match between Corollary \ref{Cor:AsympDistri}  and the Monte-Carlo simulations. 
For a larger value of $\epsilon \ge 0.01$, Corollary \ref{Cor:AsympDistri}  starts to diverge from the Monte-Carlo simulation (Figure \ref{VarianceNumerics}c), which tends to underestimate the conditional variance of $\tau_+$ as $v_0$ becomes larger. In this regard, the Monte-Carlo simulations confirms the validity of Corollary \ref{Cor:AsympDistri}  for small limit of $\epsilon \le 10^{-3}$. This result gives us information about what regime of random noise that the theoretical estimation could provide an meaningful dependence of $Var(\tau_{+}\,|\,\tau_+<\tau_0)$ on $v_0$. 

From the practical perspective, the fact that the variability of RI onset timing decreases rapidly for an initially stronger intensity (i.e., a larger value of $v_0$) would suggest that our ability to predict RI onset will be improved as TCs become stronger. This accords with previous observational and modelling studies \cite{TaoJiang2015,Fischer_etal18}, which showed indeed an overall improved RI forecasts as TCs become more organized. From this perspective, Theorem \ref{T:AsympDistri}  is anticipated and useful for further examination of the dependence of $\tau_+$ as well as its variance on different model parameters without the requirement of intensive Monte-Carlo simulations.    
%
%
\subsection{Model parameter dependence}
Given the validity domain of Corollary \ref{Cor:AsympDistri}  as established in the previous section, one can now use the explicit expression for the conditional variance of $\tau_+$ in Corollary \ref{Cor:AsympDistri}  to study how the uncertainties of RI onset time varies with different model parameters and/or initial conditions. This information is substantial, because it can help forecasters estimate the uncertainties of their RI onset prediction in real-time forecast.

Recall however that the dependence of $\eqref{Formula_Var}$ on different parameters is most useful if an estimation of the deterministic RI onset time $T$, say from a numerical or a statistical model, is given. As a result, Figure  \ref{CriticalTimeT} show the deterministic onset time $T$ for different initial condition $v_0$ and model parameters $(p,r,s)$. Here, the same hitting level $\ell=0.1$ at which the RI onset is considered to occur as $v$ crosses $\ell$ for the first time is used. 

As shown in Figure \ref{CriticalTimeT}a, $T$ is inversely proportional to $v_0$ as expected, which implies that RI onset will occur earlier for stronger initial intensity. When fixing TC initial condition, we note however that $T$ increases roughly linearly when the model parameters $s$ or $r$ increases. This linear relationship  indicates that a more stable troposphere or stronger radiative cooling will slow down RI onset as seen in Figure \ref{CriticalTimeT}b-c. In contrast, RI onset occurs earlier for a larger parameter $p$ (Figure \ref{CriticalTimeT}d), suggesting  a bigger storm size would requires less time for RI to take place. These behaviors can be used to validate our results, using observational data or modelling output that we will present in our future study.  
%
%
\begin{figure}[tbh]
\centering
\includegraphics[scale=0.5]{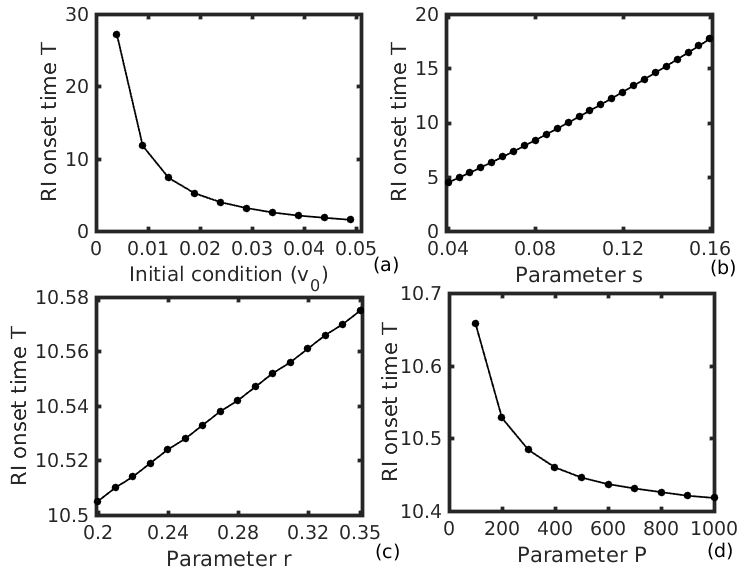}
\caption{Dependence of the deterministic RI onset time $T$ on (a) the initial condition $v_0$, (b) the atmospheric static stability parameter $s$, (c) the radiative cooling parameter $r$, and (d) the aspect ratio of the tropospheric depth over the radius of maximum wind $p$. Note that for each parameter curve, all other parameters are fixed at the values of $p=200$, $s=0.1, r=0.25$, $u_0=-0.01, v_0=0.01, b_0=0.0001$.}
\label{CriticalTimeT}
\end{figure}
%
%
\begin{figure}[tbh]
\centering
\includegraphics[scale=0.5]{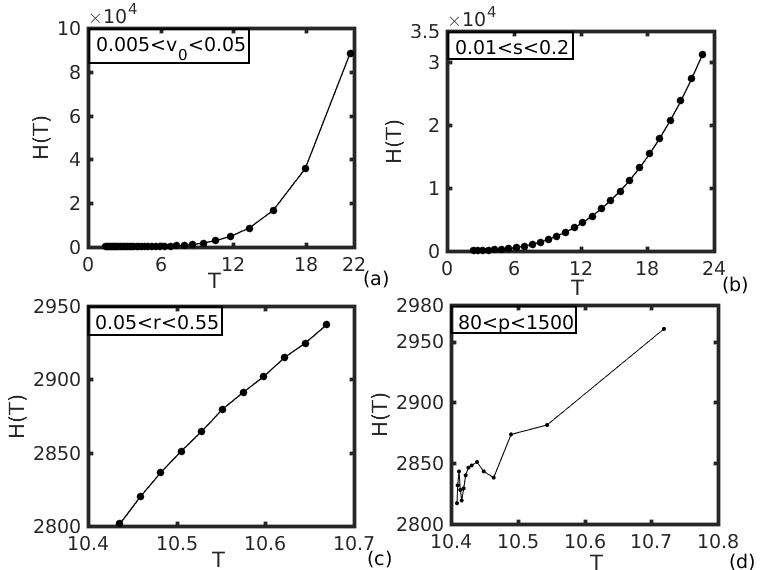}
\caption{Similar to Figure \ref{CriticalTimeT} but for the function $H(T)$ in the variance formula of Corollary \ref{Cor:AsympDistri} .}
\label{HvsParameter}
\end{figure}

Note that among all the model parameters, the conditional variance of $\tau_+$, which is represented by the function $H(T) \equiv \frac{\Sigma_{22}(T)}{\ell^2\,[u(T)+\ell]^2}$, appears to be the least sensitive to changes in the parameter $p$. On the other hand, the conditional variance of $\tau_+$ tends to be sensitive to both $r$ or $s$ (Figure.~\ref{CriticalTimeT}c-d). This sensitivity of $H(T)$ to $r$ and $s$ accords with real TC development, thus providing further understanding into the large-scale environmental factors that could affect RI onset variability, for which predictive models of RI onset must take into account in future implementation. 
%
%
\section{Conclusion}
In this paper, the rapid intensification (RI) process during tropical cyclone (TC) development was examined, using the first hitting time and asymptotic analysis for stochastic systems. By extending the TC-scale dynamical model (MSD) for TC development proposed by \refcite{Kieu2015}, RI can be considered as a random process whose onset time possesses a specific probability distribution dictated by TC dynamics. The reduced dynamics of the MSD model in the phase space of three state components $(u,v,b)$ makes it especially attractive for RI examination, as one can obtain  analytical results that could not be obtained otherwise with full-physics models.  


Specifically, by defining RI onset time as the first moment that TC intensity hits a given level $\ell$, a formal procedure to derive the RI onset probability $p(x)$ was obtained in Lemma \ref{Lm:HittingProb} through a boundary value problem. While the explicit expression for $p(x)$ has to be relied on numerical integration, its asymptotic limit $\epsilon \rightarrow 0$ could indicate that the probability of RI onset will be ensured for all initial conditions with $v_0>0$, consistent with previous modelling studies of TC development.          

Conditioned on the RI onset occurrence, we show that the timing for RI onset ($\tau_+$) would be on average longer for weaker vortex initial condition $v_0$. Also, RI onset timing will have larger variance (uncertainty) when the stochastic amplitude $\epsilon$ increases, with an asymptotic variance formula given by Corollary \ref{Cor:AsympDistri} in the small noise regime. In this small noise regime, we also demonstrated in Figure \ref{Histogram} that a larger random forcing $\epsilon$  would potentially imply a smaller probability for RI onset and a  smaller conditional expectation for RI onset time.  The latter observation is important because it helps alert forecasters a possible RI onset taking place quicker in the presence of stronger random fluctuation.

The main mathematical result regarding the variability of the RI onset time $\tau_+$ is provided by Corollary \ref{Cor:AsympDistri}, which presents an asymptotic formula for the conditional variance of $\tau_+$ in the small noise regime. Detailed examination of this variance formula using our efficient algorithm to numerically compute it from any given initial state showed that the variability of RI onset timing depends critically on TC initial intensity as well as model parameters. For a fixed set of model parameters, the variance of RI onset time decreases with initial intensity $v_0$. That is, an initially stronger vortex would experience not only earlier RI onset time but also less uncertainty in the prediction of the timing of RI onset. Similarly, the uncertainties in RI onset time will be smaller when the key model parameters such as  atmospheric stability ($s$) or the aspect ratio ($p$) decreases, suggesting a strong dependence of the RI onset forecast on the atmospheric large-scale condition.        

To examine the domain of validity of our theoretical results, Monte-Carlo simulations of the MSD system were also conducted, using the same set of parameters and initial conditions as those obtained from the theoretical analyses. Our examination of these Monte-Carlo simulations for different asymptotic limits of random noise amplitude $\epsilon$ confirmed the validity of the theoretical results for the limit of $\epsilon \rightarrow 0$. These simulations helped verify several hypotheses that were assumed in our lemmas and theorems, thus providing a broad picture of what limits our theoretical results can be applied in real TC systems.     

From the mathematical perspective, it should be noted that several results on RI onset probability and timing obtained from Monte-Carlo simulations can be intriguingly understood by using a very generic one-dimensional (1D) stochastic system. Our analyses of a general 1D stochastic equation could in fact capture well key properties of the probability distribution of RI onset as well as the timing of RI onset. In this regard, these analyses suggest that the method and results in this study can be readily applied to a more general stochastic system that possesses a first hitting time characteristic, so long as the evolution of the system prior to a rapid change in the system can be considered as a slow process. Further exploration of the first hitting time for a general 1D stochastic system will be presented in our future work. 
%
%

\newpage
\appendix
\section*{APPENDIX: Proofs for general 1-dimensional diffusions}\label{S:proof}
In this section, we provide the proofs of our results in Section \ref{S:1-dim}. The following asymptotic properties for the error function  will be useful in several places in our proofs:
For all $c\in(0,\infty)$, as $\epsilon\to 0$, we have
\begin{equation}\label{ErrorCasesAlpha}
\erf(c\,\epsilon^{\alpha-1}) \sim
\begin{cases}
1&\quad\text{if }\alpha \in(0,1) \\
 \erf(c\,)&\quad\text{if }\alpha = 1\\
\frac{2c\,}{\sqrt{\pi}} \epsilon^{\alpha-1}&\quad\text{if }\alpha\in(1,\infty)
\end{cases}.
\end{equation}
Here $A\sim B$ means $\lim_{\epsilon\to 0}A/B=1$.

\medskip

\begin{proof}[Proof of Lemma \ref{Lm:HittingProb_1d}]
Let $h(x)= \P_x(\tau_{\ell}<\tau_0)$. Then as in the proof of Lemma \ref{Lm:HittingProb}, the function $h$ 
satisfies the following boundary value problem
\begin{align*}
    \frac{\epsilon^2}{2}h''(x)+F(x)h'(x)=0 \qquad & \text{if}\quad 0<x<\ell\\
    h(\ell)=1 \quad \text{and}\quad h(0)=0.
\end{align*}
Upon solving this equation for $h$ using the  integrating factor $k_{\epsilon}$ given by \eqref{Def:kepsilon}, we obtain the desired formula.
\end{proof}

\medskip

\begin{proof}[Proof of Theorem \ref{T:asymP}]
This result and the proof is  similar to that of \cite[Lemma 6]{mcloone2018stochasticity}, which is a variation of the Laplace method. For all $t\in(0,\infty)$ there exists $\xi_t \in (0,t)$ such that
\begin{equation}\label{TaylorExpforF}
    F(t)= F'(0)t+ \frac{F''(\xi_t)}{2}t^2.
\end{equation}
To simplify notation we introduce 
\begin{align}
f_1(y)&:=\int_0^yF(t)\,dt\label{f1}\\
    f_2(y)&:=\int_0^y F'(0)t \,dt = F'(0)\frac{y^2}{2}\label{f2}\\
 I(x)&:= \int_0^{x} \exp \left(\frac{-2}{\epsilon^2}f_1(y)\right)dy\notag\\
    \widetilde{I}(x)&:= \int_0^{x} \exp \left(\frac{-2}{\epsilon^2}f_2(y)\right)dy.\notag
\end{align}

By Lemma \ref{Lm:HittingProb_1d},
\begin{equation}\label{E:HittingProb_1d}
    \P_{c\,\epsilon^{\alpha}}(\tau_{\ell}<\tau_0)
    =\frac{I(c\,\epsilon^{\alpha})}{I(\ell)}.
\end{equation}

For all $y>0$ we have,
$|f_1(y)-f_2(y)|=\left|\int_0^yF(t)\,dt-\int_0^yF'(0)t \,dt\right|\leq
      \frac{|F''(\xi_t)|}{6} y^3$.
Hence for $y<c \epsilon^\alpha$,
\begin{align}\label{f1-f2 estimate}
    |f_1(y)-f_2(y)| \leq \frac{|F''(\xi_t)|}{6} c^3\epsilon^{3\alpha}.
\end{align}
Since $F$ is smooth, $M:=\displaystyle\sup_{y\in[0,c \epsilon^{\alpha}]}|F''(y)|<\infty$. Therefore
\begin{align*}
   \left|I(c\,\epsilon^\alpha)-\widetilde{I}(c\,\epsilon^\alpha) \right|
   &=\left|\int_{0}^{c\,\epsilon^\alpha} \exp{\left(\frac{-2}{\epsilon^2}f_1(y)\right)}-\exp{\left(\frac{-2}{\epsilon^2}f_2(y)\right)}dy\right|\\
   &= \left|\int_{0}^{c\,\epsilon^\alpha} \int_{[f_1(y),f_2(y)]} \frac{-2}{\epsilon^2} e^{\frac{-2}{\epsilon^2}x}\,dx\,dy\right| \\
   &\leq \frac{2}{\epsilon^2} \int_0^{c\epsilon^\alpha}\sup_{[f_1(y),\,f_2(y)]} e^{\frac{-2}{\epsilon^2}x} \cdot |f_1(y)-f_2(y)| \,dy \\
   &\leq \frac{c^4}{3} \, M \epsilon^{2(2\alpha -1)}
\end{align*}
where in the last step we used \eqref{f1-f2 estimate}. From this we have
\begin{equation}\label{ErrorItI}
\left|\frac{I(c\,\epsilon^\alpha)-\widetilde{I}(c\,\epsilon^\alpha)}{I(\ell)}\right|\leq \frac{c^4}{3}\, M \frac{\epsilon^{4\alpha -3}}{\epsilon^{-1}I(\ell)}.
\end{equation}

For the integral $I(\ell)$ in the denominator, note that $g(y):=-2f_1(y)$ is decreasing in $[0,\ell]$ and thus it has maximum at $c=0$ the $g(0)=0$. Note also that $g'(0)=-2F(0)=0$. By Laplace method,  as $\epsilon\to 0$
\begin{align}\label{LaplaceExpansionDenom}
    I(\ell)= \int_{0}^{\ell} \exp \left(\frac{-2}{\epsilon^2}f_1(y)\right)\,dy \sim
    \sqrt{\frac{\pi}{F'(0)}} \frac{\epsilon}{2}.
\end{align}

From this and \eqref{ErrorItI}, we see that for $\alpha>\frac34$, 
\begin{equation}\label{LimitEqn}
    \lim_{\epsilon \rightarrow 0} \frac{I(c\,\epsilon^\alpha)}{I(\ell)}=\lim_{\epsilon \rightarrow 0}\frac{\widetilde{I} (c\,\epsilon^\alpha)}{I(\ell)}.
\end{equation}


By the change of variable $u=\sqrt{F'(0)} \frac{y}{\epsilon}$,
\begin{align*}
   \widetilde{I}(c\,\epsilon^\alpha)
   &=\int_0^{c\,\epsilon^\alpha} \exp{\left( \frac{-F'(0)}{\epsilon^2} y^2\right)}\,dy\\
   &=\frac{\epsilon}{\sqrt{F'(0)}} \int_0^{c\, \sqrt{F'(0)}\epsilon^{\alpha-1}} e^{-u^2} \,du\\
   &=\frac{\epsilon}{2}\sqrt{\frac{\pi}{F'(0)}} \erf (c\, \sqrt{F'(0)}\epsilon^{\alpha-1}).
\end{align*}
So for $\alpha>\frac34$, as $\epsilon\to 0$, by \eqref{LaplaceExpansionDenom} we have
\begin{align}\label{E:HittingProb_1d_2}
    \frac{I(c\,\epsilon^\alpha)}{I(\ell)}
    \sim \frac{\widetilde{I}(c\,\epsilon^\alpha)}{I(\ell)}
    = 
     \frac{\frac{\epsilon}{2}\sqrt{\frac{\pi}{F'(0)}}}{I(\ell)}  \erf(c\, \sqrt{F'(0)}\epsilon^{\alpha-1})\sim  
    \erf(c\, \sqrt{F'(0)}\epsilon^{\alpha-1})
\end{align}
where in the last step, we used \eqref{LaplaceExpansionDenom}.
From \eqref{ErrorCasesAlpha}, \eqref{E:HittingProb_1d} and \eqref{E:HittingProb_1d_2}, we obtain the desired equality for $\alpha>3/4$.

The remaining case $\alpha\in (0,3/4]$ is covered, because
 $\P_x(\tau_{\ell}<\tau_0)$ is monotonically increasing in $x$ and $\epsilon^{\alpha_1}<\epsilon^{\alpha_1}$ if $\alpha_1>\alpha_2>0$.  The proof is complete.
\end{proof}

\medskip

\begin{proof}[Proof of Lemma \ref{Lem:ConEXP 1-d}]
Let  $\tau=\min\{\tau_{\ell},\,\tau_0\}=\inf\{t\geq 0:\,Z_t=0 \text{ or } \ell\}$  be the time to exit the interval $(0,\ell)$. We shall show that
for any starting point $x\in[0,\ell]$,
\begin{equation}\label{E:RestrictedEXP 1-d}
\E_{x}[\tau 1_{\{\tau_\ell<\tau_0\}}]=\frac{2}{\epsilon^2} \frac{1}{\int_0^\ell k_\epsilon(z) dz} 
   \left\{ \int_x^\ell \int_0^x k_\epsilon(z) k_\epsilon(u) \int_u^z  \frac{p(y)}{k_\epsilon(y)}\, dy \,du \,dz\right\}.
\end{equation}

Recall that $p(x)=\P_x(\tau_\ell<\tau_0)$ and that by Lemma \ref{Lm:HittingProb_1d},
$p(x)=\frac{\int_{0}^x k_\epsilon(y)\,dy}{\int_{0}^\ell k_\epsilon(y)\,dy}$. 
The function $H$ defined by
$H(x)=\E_{x}[\tau \,1_{\{\tau_\ell<\tau_0\}}]$
solves the boundary value problem
 \begin{align}
   p(x) + \frac{\epsilon^2}{2}H''(x)+F(x)\,H'(x)&=0\qquad \text{if}\quad 0<x<\ell \label{Expected H1}\\
    H(0)=H(\ell)&=0\label{Expected H2}.
\end{align}

We now solve \eqref{Expected H1}-\eqref{Expected H2} to obtain \eqref{E:RestrictedEXP 1-d}.
Note that \eqref{Expected H1} is a first order equation in $H'(x)$, given by
\begin{align}
    H''(x)+\frac{2}{\epsilon^2}F(x)H'(x)=\frac{-2}{\epsilon^2}p(x).
\end{align}
Multiply both sides by the integrating factor $k_\epsilon^{-1}(x)=\exp{\left\{\int_0^x \frac{2}{\epsilon^2}F(t)\,dt \right\}}$,
    \begin{align*}
         [H'(x)k_\epsilon^{-1}(x)]'&=\frac{-2}{\epsilon^2} k_\epsilon^{-1}(x)p(x)\\
        H'(x)&=\frac{-2}{\epsilon^2} k_\epsilon(x)K(x) +Ck_\epsilon(x),    
    \end{align*}
where we let $K(x)=\int_0^x k_\epsilon^{-1}(y)p(y)\, dy$ for simplicity.
    
Integrating again and using the fact $H(0)=0$, we have
\[
 H(x)-0=-\frac{2}{\epsilon^2}\int_0^x k_\epsilon(z)\,K(z) \,dz +C\int_0^x k_\epsilon(z) \,dz.
\]
To compute $C$ we use the fact that $H(\ell)=0$. We find that
\[C=\frac{\frac{2}{\epsilon^2}\int_0^\ell k_\epsilon(z)\,K(z) \,dz}{\int_0^\ell k_\epsilon(z) dz}.
\]
From the last two displayed equations,
\begin{align}
   H(x)=&-\frac{2}{\epsilon^2}\int_0^x k_\epsilon(z)\,K(z) \,dz +  \frac{\frac{2}{\epsilon^2}\int_0^\ell k_\epsilon(z)\,K(z) \,dz}{\int_0^\ell k_\epsilon(z) dz} \cdot \int_0^x k_\epsilon(z) \,dz \notag\\
    =&\frac{2}{\epsilon^2} \left\{ \int_0^\ell k_\epsilon(z)\,K(z) \,dz \cdot \frac{\int_0^x k_\epsilon(z) \,dz}{\int_0^\ell k_\epsilon(z) dz} - \int_0^x k_\epsilon(z)\,K(z) \,dz   \right\} \notag\\ 
    =&\frac{2}{\epsilon^2} \left\{ \int_0^\ell k_\epsilon(z)\,K(z) \,dz \cdot p(x) - \int_0^x k_\epsilon(z)\,K(z) \,dz   \right\} \label{Expectationsubset}
\end{align}

 We now rewrite \eqref{Expectationsubset} in a way that reflects why the complicated expression on the right is non-negative.
 \begin{align*}
   H(x)&=\frac{2}{\epsilon^2} \left\{ \int_0^\ell k_\epsilon(z)\,K(z) \,dz \cdot \frac{\int_0^x k_\epsilon(z) \,dz}{\int_0^\ell k_\epsilon(z) dz} - \int_0^x k_\epsilon(z)\,K(z) \,dz   \right\}\\
   &=\frac{2}{\epsilon^2} \frac{1}{\int_0^\ell k_\epsilon(z) dz} \left\{ \int_0^\ell k_\epsilon(z) K(z)\,dz  \int_0^x k_\epsilon(z) \,dz- \int_0^\ell k_\epsilon(z) dz \int_0^x k_\epsilon(z) K(z) \,dz   \right\}\\
   &=\frac{2}{\epsilon^2} \frac{1}{\int_0^\ell k_\epsilon(z) dz} 
   \left\{ \int_x^\ell \int_0^x k_\epsilon(z) k_\epsilon(u) (K(z)-K(u)) \,du \,dz\right\},
   \end{align*}
   where in the last equality we have used the fact (by symmetry) that $\int_0^x\int_0^x k_\epsilon(z)k_\epsilon(u)\left(K(z)-K(u)\right)\,du \,dz=0$.

In conclusion, we proved \eqref{E:RestrictedEXP 1-d}. The lemma then follows from \eqref{E:RestrictedEXP 1-d}.
\end{proof}

\begin{proof}[Proof of Theorem \ref{T:AsymTime}]
Recall the formula \eqref{E:ConEXP 1-d} in Lemma \eqref{Lem:ConEXP 1-d}.
Fix $\epsilon>0$ and let $x\to 0$, the denominator in \eqref{E:ConEXP 1-d} is of order $x$ in the sense that
\begin{equation}\label{D_xto0}
\lim_{x\to 0}\frac{\int_0^x k_\epsilon(z) dz}{x}=k_\epsilon(0)=1.
\end{equation}
The numerator of \eqref{E:ConEXP 1-d} is also of order $x$ in the sense that, 
\begin{align}\label{N_xto0}
\lim_{x\to 0}\frac{1}{x}\int_x^\ell \int_0^x k_\epsilon(z) k_\epsilon(u) \int_u^z  \frac{p(y)}{k_\epsilon(y)}\, dy \,du \,dz
=\int_0^\ell k_\epsilon(u) \int_0^u \frac{p(y)}{k_\epsilon(y)}\, dy \,du.
\end{align}
Equation \eqref{N_xto0} follows from the L'Hospital rule and the Leibniz integral rule as follows. Define
\begin{align*}
L(x)&:=\int_x^\ell \int_0^x k_\epsilon(z) k_\epsilon(u) \int_u^z  \frac{p(y)}{k_\epsilon(y)}\, dy \,du \,dz\\ 
    F_1(x,z)&:=\int_0^x k_\epsilon(z) k_\epsilon(u) \int_u^z  \frac{p(y)}{k_\epsilon(y)}\, dy \,du\\
    G_1(z,u)&:=k_\epsilon(z)k_\epsilon(u) \int_u^z  \frac{p(y)}{k_\epsilon(y)}\, dy.
\end{align*}
Then $L(x)=\int_x^\ell F_1(x,z) \,dz $ and $F_1(x,z)=\int_0^x G_1(z,u) \,du$.
By Leibniz integral rule,
 \begin{align*}
L'(x)
&= -F_1(x,x) + \int_x^\ell \frac{\partial}{ \partial x}F_1(x,z) \,dz\\
&= -k_\epsilon(x) \int_0^x  k_\epsilon(u) \int_u^x  \frac{p(y)}{k_\epsilon(y)}\, dy \,du + \int_x^\ell G_1(z,x) \,dz\\
&= k_\epsilon(x) \int_0^x  k_\epsilon(u) \int_x^u  \frac{p(y)}{k_\epsilon(y)}\, dy \,du + k_\epsilon(x) \int_x^\ell k_\epsilon(u) \int_x^u \frac{p(y)}{k_\epsilon(y)}\, dy \,du\\
&=k_\epsilon(x) \left\{ \int_0^\ell k_\epsilon(u) \int_x^u \frac{p(y)}{k_\epsilon(y)}\, dy \,du\right\}.
\end{align*}
Letting $x\to 0$ and applying L'Hospital rule, we obtain \eqref{N_xto0}.
By \eqref{D_xto0} and \eqref{N_xto0}, the proof of \eqref{E:AsymTime1} is complete.

The second and the third derivatives of $L(x)$ are
\begin{align*}
L''(x)= &k_\epsilon'(x)\int_0^\ell k_\epsilon(u) \int_x^u \frac{p(y)}{k_\epsilon(y)}\, dy \,du- p(x)\int_0^\ell k_\epsilon(u)\,du\\
L^{(3)}(x)= &k_\epsilon''(x)\int_0^\ell k_\epsilon(u) \int_x^u \frac{p(y)}{k_\epsilon(y)} dy du -  \frac{k_\epsilon'(x)p(x)}{k_\epsilon(x)}\int_0^\ell k_\epsilon(u) du- p'(x) \int_0^\ell k_\epsilon(u)du.
\end{align*}

Note that $k'_\epsilon(x)=\frac{-2}{\epsilon^2} F(x)\,k_\epsilon(x)$. So $F(0)=0$ implies $k'_\epsilon(0)=0$. Furthermore,  $k''_\epsilon(0)=\frac{-2}{\epsilon^2}F'(0)<0$ since $F'(0)>0$. These give $L''(0)=0$ and
\[
L^{(3)}(0)=\frac{-2F'(0)}{\epsilon^2}\,C-1 \,<\,0, \quad\text{where }C=\int_0^\ell k_\epsilon(u) \int_0^u \frac{p(y)}{k_\epsilon(y)}\, dy \,du=L'(0).
\]
Note that
\begin{align*}
\E_{x}[\tau_{\ell}\,|\,\tau_{\ell}<\tau_0]-\Psi(\epsilon)
=\, \frac{2}{\epsilon^2} \left(\frac{L(x)}{\int_0^x k_\epsilon(z) dz} - L'(0)\right).
\end{align*}
Hence 
\begin{align*}
\lim_{x\to 0}\frac{\E_{x}[\tau_{\ell}\,|\,\tau_{\ell}<\tau_0]-\Psi(\epsilon)}{x^2}
=&\, \frac{2}{\epsilon^2} \lim_{x\to 0}\frac{L(x) - L'(0)\int_0^x k_\epsilon(z) dz}{x^2\,\int_0^x k_\epsilon(z) dz}\\
=&\, \frac{2}{\epsilon^2} \lim_{x\to 0}\frac{L'(x) - L'(0)\, k_\epsilon(x)}{x^2 k_\epsilon(x) + 2x\int_0^x k_\epsilon(z) dz}\\
=&\, \frac{2}{\epsilon^2} \lim_{x\to 0}\frac{L''(x) - L'(0)\, k'_\epsilon(x)}{x^2 k_\epsilon'(x) +  4xk_\epsilon(x) + 2\int_0^x k_\epsilon(z) dz}\\
=&\, \frac{2}{\epsilon^2} \lim_{x\to 0}\frac{L^{(3)}(x) - L'(0)\, k''_\epsilon(x)}{x^2 k_\epsilon''(x)+ 6xk'_\epsilon(x)+6k_\epsilon(x) }\\
=&\, \frac{2}{\epsilon^2}\frac{L^{(3)}(0) - L'(0)\, k''_\epsilon(0)}{6}\\
=&\, \frac{1}{3\epsilon^2}\left(\frac{-2F'(0)}{\epsilon^2}\,C-1\right) + C\, \frac{2F'(0)}{\epsilon^2}\\
=&\,\frac{-1}{3\epsilon^2}.
\end{align*}
\end{proof} 

\section*{Acknowledgment}
This research was partially supported by the ONR/DRI Award N000142012411. DS wish to also thank ``Andreas Mentzelopoulos Scholarships University of Patras" for their financial support during his doctoral program at Indiana University. 
\bibliographystyle{ws-m3as}
\bibliography{ws-m3as}
\end{document}